\documentclass{amsart}
\usepackage{bbm, mathrsfs, enumerate, amssymb, fourier}
\usepackage[foot]{amsaddr}

\def\Bcal{{\mathcal{B}}}
\def\Ccal{{\mathcal{C}}}
\def\Mcal{{\mathcal{M}}}
\def\Scal{{\mathcal{S}}}

\def\E{\mathbb{E}}
\def\N{\mathbb{N}}
\def\P{\mathbb{P}}
\def\R{\mathbb{R}}

\def\offdistr{q}

\def\d{\mathrm{d}}
\def\refAone{{\normalfont (A\ref{eq:cond-A1})}}
\def\refAtwo{{\normalfont (A\ref{eq:cond-A2})}}

\newcommand{\indicator}[1]{\mathbbm{1}_{\{#1\}}}
\newcommand{\bottleneck}{\wp}

\newcommand{\tv}[1]{\lVert #1 \rVert}

\theoremstyle{plain}
\newtheorem{assumption}{Assumption}

\newtheorem{Thm}{Theorem}[section]
\newtheorem{Lem}[Thm]{Lemma}
\newtheorem{Prop}[Thm]{Proposition}

\newtheorem*{condition}{Condition}

\theoremstyle{remark}
\newtheorem{Rk}{Remark}[section]

\numberwithin{equation}{section}

\title[Scaling limits of Galton Watson processes in varying environment]{On the scaling limits of Galton Watson processes in varying environment}
\date{\today}

\author[V.\ Bansaye]{Vincent Bansaye$^1$}
\email{vincent.bansaye@polytechnique.edu}
\address{$^1$CMAP, Ecole Polytechnique, Route de Saclay\\
91128 Palaiseau Cedex, France.}

\author[F.\ Simatos]{Florian Simatos$^2$}
\email{florian.simatos@inria.fr}
\address{$^2$Inria Paris-Rocquencourt and ENS Paris, 23 avenue d'Italie, CS 81321, 75214 Paris Cedex 13, France}

\thanks{This work was funded by project MANEGE ``Mod\`eles
Al\'eatoires en \'Ecologie, G\'en\'etique et \'Evolution''
09-BLAN-0215 of ANR (French national research agency), Chair Mod\'elisation Math\'ematique et Biodiversit\'e VEOLIA-Ecole Polytechnique-MNHN-F.X.\ and the professoral chair Jean Marjoulet. While most of this work was carried out, the second author was affiliated with CWI and his research was funded by an NWO-VIDI grant.}

\begin{document}

\maketitle
\begin{abstract}
	We establish a general sufficient condition for a sequence of Galton Watson branching processes in varying environment to converge weakly. This condition extends previous results by allowing offspring distributions to have infinite variance, which leads to new and subtle phenomena when the process goes through a bottleneck and also in terms of time scales.
	
	Our assumptions are stated in terms of pointwise convergence of a triplet of two real-valued functions and a measure. The limiting process is characterized by a backwards ordinary differential equation satisfied by its Laplace exponent, which generalizes the branching equation satisfied by continuous state branching processes. Several examples are discussed, namely branching processes in random environment, Feller diffusion in varying environment and branching processes with catastrophes.
\end{abstract}

\setcounter{tocdepth}{1}

\bigskip
\hrule
\vspace{-2mm}
\tableofcontents
\vspace{-8mm}
\hrule

\newpage

\section{Introduction}

In this paper we extend previous results on scaling limits of Galton Watson branching processes in varying environment. More precisely, for each $n \geq 1$ we consider a sequence of offspring distributions $({\offdistr}_{i,n}, i \geq 0)$, the \emph{environment}, and we consider the Galton Watson process $Z_n = (Z_{i,n}, i \geq 0)$ where individuals of the $i$-th generation reproduce according to ${\offdistr}_{i,n}$. We are interested in the asymptotic behavior of the sequence $(X_n, n \geq 1)$ of rescaled processes of the form $X_n(t) = n^{-1} Z_{\gamma_n(t), n}$ for some sequence of time-changes $\gamma_n$.

In the Galton Watson case where ${\offdistr}_{i,n} = {\offdistr}_{0,n}$, this problem has been first considered by Feller~\cite{Feller51:0} and later by Lamperti~\cite{Lamperti67:0, Lamperti67:2}, and was solved by Grimvall~\cite{Grimvall74:0}. In this case, the assumptions that we make for our main result are equivalent to Grimvall's necessary and sufficient ones. Moreover, in this case, the structure of the possible limit processes, called continuous state branching processes (CSBP) and first considered by Ji{\v r}ina~\cite{Jirina58:0}, is well understood thanks to a random time-change transformation exhibited by Lamperti~\cite{Lamperti67:1}, and sometimes called Lamperti transformation.

In the case of varying environment, results are to our knowledge much scarcer and the main results seem to be due to Kurtz~\cite{Kurtz78:0} and Borovkov~\cite{Borovkov02:0}, to which our main theorem will be compared in details in Section~\ref{sub:comparison}. The two main points are that: (1) on the upside, we extend these results to the case of offspring distributions with (possibly) infinite variance; (2) on the downside, we assume that a certain function has locally finite variation. This finite variation assumption may seem restrictive, but it is \emph{intrinsic to our approach}: if it does not hold, fundamentally different techniques than the ones we present here are needed in order to characterize accumulation points of the sequence $(X_n, n \geq 1)$, see the discussion in Section~\ref{sub:comparison}.
\\

In the case ${\offdistr}_{i,n} = {\offdistr}_{0,n}$ of constant environment, an elegant way to study the asymptotic behavior of a sequence of rescaled Galton Watson processes is to extend the Lamperti transformation at the discrete level and leverage results on the convergence of random walks and on the continuity of random time-change maps, see for instance Ethier and Kurtz~\cite[Chapter~$9$]{Ethier86:0} or, in the continuous-time setting, Helland~\cite{Helland78:0}. Nonetheless, this approach breaks down when offspring distributions vary from one generation to the other, as we discuss in Section~\ref{sub:comparison}.

When offspring distributions vary but have finite variance, the authors in~\cite{Borovkov02:0, Kurtz78:0} express their limit process as a simple transformation of Feller diffusion, the only CSBP with continuous sample paths. This allows them to leverage results for continuous diffusion processes. Kurtz~\cite{Kurtz78:0} for instance uses semigroup techniques developed in~\cite{Kurtz75:0}. However, these techniques become significantly more demanding in the infinite variance case considered here, where one needs to consider diffusion processes with jumps.

So we use a third approach, which, similarly as in Grimvall~\cite{Grimvall74:0}, is based on the convergence of Laplace exponents. In the Galton Watson case, CSBP's can be characterized by an ordinary differential equation, called \emph{branching equation}, satisfied by their Laplace exponent, see for instance Silverstein~\cite{Silverstein67:0} or Caballero et al.~\cite{Caballero09:0} for a recent (and complete) treatment. In the present time-inhomogeneous case, we show that accumulation points can be characterized by an ordinary differential equation that generalizes the branching equation. Interestingly, this generalized branching equation is then backwards in time.

This approach makes it possible to deal with a complication that does not exist in the finite variance case studied previously. Indeed, in the infinite variance case, the limiting processes may be non-conservative, i.e., may explode in finite time. Moreover, since offspring distributions are allowed to vary over time, nothing prevents a \emph{catastrophic} environment to occur, where the mean number of children is close to $0$. When a non-conservative process goes through such a bottleneck, we (intuitively) run into an indetermination of the kind $\infty \times 0$. To deal with this complication we introduce a new notion of \emph{bottleneck} which plays a crucial role throughout the analysis. This peculiar phenomenon is discussed in more depth in Section~\ref{sub:bottleneck}.
\\

Let us now mention some closely related results. Galton Watson processes in \emph{random} environment were first introduced and studied in Smith and Wilkinson~\cite{Smith69:0} in the case where the sequence $({\offdistr}_{i,n}, i \geq 0)$ is i.i.d., and in Athreya and Karlin~\cite{Athreya71:1, Athreya71:0} when this sequence is stationary. These models have recently attracted considerable interest in the literature, see for instance~\cite{Afanasyev10:1, Afanasyev05:0, Bansaye09:0, Birkner05:0, Boinghoff10:0, Boinghoff12:0, Geiger03:0, Guivarch01:0} for results on the long-time behavior in the critical and subcritical regimes and on large deviation. Scaling limits in the finite variance case were conjectured by Keiding~\cite{Keiding75:0} who introduced Feller diffusion in random environment. This conjecture was proved by Kurtz~\cite{Kurtz78:0} and Helland~\cite{Helland81:0}. In particular, Kurtz' results establish \emph{quenched} results since, as mentioned above, his results apply to a deterministic sequence of varying offspring distributions. In the same way, our results apply to an almost sure realization of an i.i.d.\ sequence of offspring distributions. Since these offspring distributions may have infinite variance, we end up with a process more general, in some aspects, than Feller diffusion in random environment. We describe the probabilistic structure of this process in Section~\ref{subsub:GWRE}, and we believe that this constitutes an interesting object for future work. Moreover, our results shed light on a subtle question related to the correct renormalization in time of such processes which is discussed in Section~\ref{subsub:GWRE}.

Our results are also related to some results on superprocesses. More precisely, our limit processes are closely related to the mass of superprocesses considered in El Karoui and Roelly~\cite{El-Karoui91:0}. Under some additional technical assumptions, e.g., finite first moment and no drift, these superprocesses are obtained in Dynkin~\cite{Dynkin91:0, Dynkin93:0} as the limit of suitable branching particle systems. However, in these works the emphasis is on the limiting superprocesses themselves. As such, Dynkin~\cite{Dynkin91:0, Dynkin93:0} considers branching particle systems evolving in continuous time, which, in order to establish limit theorems, are technically more convenient than the discrete time setting considered here. Our motivation here is different, since we want to understand the asymptotic behavior of Galton Watson processes, even in the non-conservative case and in the presence of a drift term.

\subsection*{Organization of the paper} Theorem~\ref{thm:main} is the main result of the paper, and is presented in Section~\ref{sub:main-result}. We compare it with earlier results in Section~\ref{sub:comparison} and discuss some applications in Section~\ref{sub:applications}, namely to Galton Watson processes in random environment, to Feller diffusion in varying environment and to CSBP with catastrophes. Section~\ref{sec:additional-notation-preliminary-results} introduces further notation that are used throughout the rest of the paper, as well as some preliminary results. Theorem~\ref{thm:main} is proved in Section~\ref{proofThm1}, with some technical proofs deferred to Appendices~\ref{appendix:constants} and~\ref{appendix:proof-lemma}. Further results that complement Theorem~\ref{thm:main} are proved in Section~\ref{sec:proof-propositions}, and Appendix~\ref{appendix:proof-GW} is devoted to checking that the assumptions of Theorem~\ref{thm:main} are necessary and sufficient in the Galton Watson case.

\subsection*{Acknowledgements.} We thank Julien Berestycki for interesting discussions at the earliest stage of this work.

\section{Notation and results} \label{sec:notation}

\subsection{General notation}

We first introduce the minimal set of notation needed in order to state our results, the rest of the notation will be introduced in Section~\ref{sub:additional-notation}. In the rest of the paper, if a function $g$ defined on $[0,\infty)$ is c\`adl\`ag, we implicitly assume that it is locally bounded and we write $\Delta g(t) = g(t) - g(t-)$ for the value of the jump of $g$ at time~$t$. If $g$ is in addition of locally finite variation, we write $\tv{g}(t)$ for the total variation of $g$ on $[0,t]$ and $\int f d g$ for the Lebesgue-Stieltjes integral of a measurable function $f$; note that $\lvert\int f dg\rvert \leq \int \lvert f\rvert d\tv{g}$. Moreover, we say that a function $f$ is increasing if $f(x) \geq f(y)$ for every $x \geq y$.
\\

For each $n \geq 1$, we consider a Galton Watson process in varying environment $Z_n = (Z_{i,n}, i \geq 0)$. We denote by ${\offdistr}_{i,n}$ the offspring distribution in generation $i$ and $\xi_{i,n}$ a random variable distributed according to ${\offdistr}_{i,n}$, so that we can construct $Z_n$ according to the following recursion:
\[ Z_{i+1,n} = \sum_{k=1}^{Z_{i,n}} \xi_{i,n}(k), \ i \geq 0, \]
where the random variables $(\xi_{i,n}(k), i, k \geq 0)$ are independent and $\xi_{i,n}(k)$ is equal in distribution to $\xi_{i,n}$. In order to find an interesting renormalization of the sequence of processes $(Z_n, n \geq 1)$, we fix the space scale equal to $n$ while the time scale is allowed to vary over time. More precisely, for $n \geq 1$, we consider an increasing, c\`adl\`ag and onto function $\gamma_n : [0,\infty) \to \N$ (here and elsewhere, $\N = \{0, 1, \ldots\}$ denotes the set of non-negative integers) and we define the renormalized process $(X_n(t), t \geq 0)$ as follows:
\[ X_n(t) = \frac{1}{n} Z_{\gamma_n(t),n}, \ t \geq 0. \]

For $i \geq 0$ and $n \geq 1$, we define $t_i^n = \inf\{ t \geq 0: \gamma_n(t) = i \}$ so that $\gamma_n(t_i^n) = i$ and $t_{\gamma_n(t)}^n \leq t < t_{\gamma_n(t)+1}^n$. Since $Z_n$ satisfies the branching property, i.e., $Z_n$ started from $Z_{0,n} = z$ is stochastically equivalent to the sum of $z$ i.i.d.\ processes distributed according to $Z_n$ started from $Z_{0,n} = 1$, we obtain after scaling
\begin{equation} \label{eq:def-u}
	\E\left[ \exp\left(-\lambda X_n(t) \right) \ \vert \ X_n(s) = x \right] = \exp(-x u_n(s,t,\lambda))
\end{equation}
for all $\lambda, x, s, t \geq 0$ with $s \leq t$ and $u_n(s,t,\lambda) \geq 0$ called the Laplace exponent. We will characterize the convergence of $X_n$ through the convergence of $u_n$. Our assumptions will be stated in terms of the convergence of the triplet $(\alpha_n, \beta_n, \nu_n)$, where $\alpha_n$ and $\beta_n$ are real-valued functions and $\nu_n$ is a measure on $\R \times [0,\infty)$, see Assumption~\ref{assumptions} below. This triplet is defined in terms of the renormalized random variables
\[ \overline \xi_{i,n} = \frac{1}{n} \left(\xi_{i,n} - 1\right), \ i \geq 0, n \geq 1, \]
in the following way: for $t \geq 0$,
\[ \alpha_n(t) = n \sum_{i=0}^{\gamma_n(t)-1} \E \left( \frac{\overline \xi_{i,n}}{1+\overline \xi_{i,n}^2} \right) \ \text{ and } \ \beta_n(t) = \frac{1}{2} n \sum_{i=0}^{\gamma_n(t)-1} \E \left( \frac{\overline \xi_{i,n}^2}{1+\overline \xi_{i,n}^2} \right), \]
and for $t \geq 0$ and $x \in \R$,
\[ \nu_n([x,\infty) \times (0,t]) = n \sum_{i=0}^{\gamma_n(t)-1} \P\big( \overline \xi_{i,n} \geq x \big). \]

Each time, we understand a sum of the form $\sum_0^{-1}$ to be equal to $0$, so that $\alpha_n(0) = \beta_n(0) = 0$. It will be convenient to introduce the numbers
\[ \alpha_{i,n} = n \E \left( \frac{\overline \xi_{i,n}}{1+\overline \xi_{i,n}^2} \right) \ \text{ and } \ \beta_{i,n} = \frac{1}{2} n \E \left( \frac{\overline \xi_{i,n}^2}{1+\overline \xi_{i,n}^2} \right) \]
and the measures $\nu_{i,n}([x,\infty)) = n \P( \overline \xi_{i,n} \geq x)$, so that we can write
\[ \alpha_n(t) = \sum_{i=0}^{\gamma_n(t)-1} \alpha_{i,n}, \ \beta_n(t) = \sum_{i=0}^{\gamma_n(t)-1} \beta_{i,n} \ \text{ and } \ \nu_n([x,\infty) \times (0,t]) = \sum_{i=0}^{\gamma_n(t)-1} \nu_{i,n}([x,\infty)). \]

\subsection{Main result} \label{sub:main-result}

The main result of the paper, Theorem~\ref{thm:main}, will hold under the following assumptions on $\alpha_n$, $\beta_n$ and $\nu_n$. As mentioned in the introduction, our approach relies on controlling the asymptotic behavior of the Laplace exponent $u_n$. We exhibit in~\eqref{eq:dynamics-u-n} a discrete dynamical system satisfied by $u_n$ where $\alpha_n$, $\beta_n$ and $\nu_n$ appear, which makes these assumptions natural in order to make this discrete dynamical system converge to a continuous one.

Moreover, these assumptions have the advantage, from a modeling perspective, to highlight the characteristics of the successive reproduction laws that play a key role with respect to scaling limits. These are, namely, the first and second truncated moments and the tail distributions.

Finally, in addition to help understanding scaling limits, these assumptions also make it possible to get fine results on the qualitative behavior of the limiting process (see namely Propositions~\ref{prop:no-bottleneck},~\ref{prop:no-explosion} and~\ref{prop:Feller}).

\begin{assumption} \label{assumptions}
	There exist a c\`adl\`ag function of locally finite variation $\alpha$, an increasing c\`adl\`ag function $\beta$, and a positive measure $\nu$ on $\R \times [0,\infty)$ with support included on $(0,\infty) \times (0,\infty)$, such that the two following conditions hold:
	
	\begin{enumerate}[\normalfont ({A}1)]
		\item \label{eq:cond-A1} For every $t \geq 0$ and every $x > 0$ such that $\nu(\{x\} \times (0,t])=0$,
		\begin{multline*}
			\alpha_n(t) \mathop{\longrightarrow}_{n \to +\infty} \alpha(t), \, \tv{\alpha_n}(t) \mathop{\longrightarrow}_{n \to +\infty} \tv{\alpha}(t), \, \beta_n(t) \mathop{\longrightarrow}_{n \to +\infty} \beta(t)\\
			\text{and } \, \nu_n([x, \infty) \times (0,t]) \mathop{\longrightarrow}_{n \to +\infty} \nu([x, \infty) \times (0,t]).
		\end{multline*}
		\item \label{eq:cond-A2} For every $t$ such that either $\Delta \alpha(t) \ne 0$, $\Delta\beta(t) \ne 0$ or $\nu((0,\infty)\times\{t\}) \ne 0$ and for every $x > 0$ such that $\nu(\{x\} \times \{t\})=0$,
		\[ \alpha_{\gamma_n(t), n} \mathop{\longrightarrow}_{n \to +\infty} \Delta \alpha(t), \, \beta_{\gamma_n(t), n} \mathop{\longrightarrow}_{n \to +\infty} \Delta \beta(t) \text{ and } \nu_{\gamma_n(t), n}([x,\infty))\mathop{\longrightarrow}_{n \to +\infty} \nu( [x,\infty)\times \{t\}). \]
	\end{enumerate}
\end{assumption}

As will be seen shortly, our limit processes can explode in finite time; this is already the case for CSBP's, see, e.g., Grey~\cite{Grey74:0}. However, because in our framework a highly subcritical offspring distribution may occur, we may run into an indetermination of the kind $\infty \times 0$ if the process passes through what will be referred to as a \emph{bottleneck}, see the example given in Section~\ref{sub:bottleneck}. In order to deal with this possible complication, our main result concerns the behavior of the process on the time interval $[\bottleneck(t), t]$, where $\bottleneck(t)$ is the time of last bottleneck before time $t$:
\begin{equation} \label{eq:def-sigma}
	\bottleneck(t) = \sup \left\{s \leq t : \lim_{\varepsilon \to 0} \liminf_{n\rightarrow\infty} \inf_{s \leq y \leq t} \P\left(X_n(t) > \varepsilon \ \big \vert \ X_n(y)=1 \right) = 0\right\}
\end{equation}
with the convention $\sup \emptyset = 0$. This definition of the bottleneck is the most technically convenient at this point: an alternative, more intuitive definition is given in Lemma~\ref{lemma:inf}.

\begin{Thm}[Behavior on $[\bottleneck(t),t\rbrack$] \label{thm:main}
	Assume that Assumption~\ref{assumptions} holds, and let $\alpha$, $\beta$ and $\nu$ the functions and measure defined there. Then, the following properties hold. 
	\begin{enumerate}[\normalfont I.]
		\item \label{property:triplet} For every $t \geq 0$, we have $\Delta \alpha(t) \geq -1$ and $\int_{(0,\infty) \times (0,t]} (1\wedge x^2) \nu(dx \, dy) < +\infty$. Moreover, the following function $\widetilde \beta$ is continuous and increasing:
		\[ \widetilde \beta(t)= \beta(t) - \int_{(0,\infty) \times (0,t]} \frac{x^2}{2(1+x^2)} \nu (dx \, dy), \ t \geq 0. \]
		
		\item \label{property:u} For every $t, \lambda > 0$ and $s \in [\bottleneck(t), t]$, there exists $u(s,t,\lambda) \in (0,\infty)$ such that 
		\[ \lim_{n \to +\infty} u_n(s,t,\lambda) = u(s,t,\lambda). \]
		
		Moreover, the function $u_{t,\lambda} : s \in [\bottleneck(t),t] \mapsto u (s,t,\lambda)$ is the unique c\`adl\`ag function that satisfies $\inf_{s \leq y \leq t} u_{t, \lambda}(y) > 0$ for every $\bottleneck(t) < s \leq t$ and
		\begin{multline} \label{eq:dynamics-u}
			u_{t,\lambda}(s) = \lambda + \int_{(s,t]} u_{t,\lambda}(y) \alpha(dy) - \int_{(s,t]} u_{t,\lambda}(y)^2 \widetilde \beta(dy) \\
			+ \int_{(0,\infty) \times (s,t]} \left( 1 - e^{-x u_{t,\lambda}(y)} - \frac{x u_{t,\lambda}(y)} {1+x^2} \right) \nu (dx \, dy)
		\end{multline}
		for every $\bottleneck(t) \leq s \leq t$.
		
		\item \label{property:fdd-convergence} Fix $t \geq 0$, $s \in [\bottleneck(t), t]$ and $x \geq 0$. Then for every sequence of initial states $(x_n)$ with $x_n \to x$, every $I \geq 1$, every $s \leq t_1 < \cdots < t_I \leq t$ and every $\lambda_1, \ldots, \lambda_I > 0$,
		\begin{multline} \label{eq:fdd}
		\lim_{n \to +\infty} \E \left[ \exp \left( -\lambda_1 X_n(t_1) - \cdots -\lambda_I X_n(t_I) \right) \mid X_n(s) = x_n \right] \\
		= \exp\bigg(-x u\Big(s,t_1,\lambda_1+u\big(t_1,t_2,\lambda_2+u(\cdots,u(t_{I-1},t_I,\lambda_I)\cdots)\big)\Big)\bigg).
		\end{multline}
		\item \label{property:tightness} Fix $t \geq 0$, $s \in [\bottleneck(t), t]$ and $x \geq 0$. Then for every sequence of initial states $(x_n)$ with $x_n \to x$, the sequence of processes $(X_n(y), s \leq y \leq t)$ under $\P(\, \cdot \mid X_n(s) = x_n)$ is tight on the space $D([s, t], [0,\infty])$ of c\`adl\`ag functions $f: [s, t] \to [0,\infty]$ endowed with the Skorohod topology, where the space $[0,\infty]$ is equipped with the metric $d(x,y) = \lvert e^{-x} - e^{-y}\rvert$. In particular, weak convergence holds in view of~\eqref{eq:fdd}.
	\end{enumerate}
\end{Thm}

In property~\ref{property:tightness} we consider $X^n$ as a process with range $[0,\infty]$. Although $X^n$ for fixed $n$ cannot explode, for technical reasons we need to specify its behavior started at $\infty$: in the sequel we assume that $+\infty$ is an absorbing state, so that if $X^n(s) = \infty$ for some $s$, then $X^n(t) = +\infty$ for all $t \geq s$. Let us make a couple of further remarks before proceeding.

\begin{Rk}
	The convergence $u_n(s,t,\lambda) \to u(s,t,\lambda)$ in property~\ref{property:u} actually holds in a functional sense, see Remark~\ref{rk:relative-compactness}.
\end{Rk}

\begin{Rk} \label{remark:skorohod}
	Property~\ref{property:fdd-convergence} is stronger than the usual notion of finite-dimensional convergence, where~\eqref{eq:fdd} typically only holds for times $t_i$ such that $\P(\Delta X(t_i) = 0) = 1$. We get such a result because~\refAone\ holds for all $t \geq 0$, and not only for those for which the functions $(\alpha(t), t \geq 0)$, $(\beta(t), t \geq 0)$ and $(\nu(\, \cdot \, \times (0,t]), t \geq 0)$ are continuous.
	
	To assume~\refAone\ even for fixed times of discontinuity may seem unusual when compared to previous results where typically one would assume convergence in the usual Skorohod topology (such as in~\cite{Borovkov02:0, Kurtz78:0}), which only implies~\refAone\ outside these times. However, if for instance $\alpha_n$ converges to $\alpha$ in a functional sense, then we can find $\gamma'_n$ such that~\refAone\ and~\refAtwo\ hold when considering $\gamma'_n$ instead of $\gamma_n$, see Section~\ref{subsub:GWRE} where this argument is detailed.
\end{Rk}

\begin{Rk}
	The proof of property~\ref{property:tightness} will actually show that $(X_n(y), s \leq y \leq t)$ under $\P(\, \cdot \mid X_n(s) = x_n)$ is tight for any $0 \leq s \leq t$, not only $\bottleneck(t) \leq s \leq t$.
\end{Rk}

\subsection{Around the bottleneck} \label{sub:bottleneck}

We now discuss in more details the notion of bottleneck. Intuitively, we expect the process to be sent to $0$ when going through a bottleneck, which, technically, would mean that $u_n(s,t,\lambda) \to 0$ if $s < \bottleneck(t)$. We now consider an example which illustrates several things that can go wrong, thus justifying our framework, in particular the introduction of the bottleneck $\bottleneck(t)$ and its role in Theorem~\ref{thm:main}.

Consider a critical offspring distribution ${\offdistr}$ and $Y_n = (Y_{i,n}, i \geq 0)$ the Galton Watson process with offspring distribution ${\offdistr}$, started from $n$ individuals. Assume that ${\offdistr}$ and $\Gamma_n$ are such that the sequence $(\widehat Y_n)$ with $\widehat Y_n(t) = Y_{\lfloor \Gamma_n t \rfloor, n} / n$ converges toward a non-conservative CSBP $\widehat Y$. For each $n \geq 1$, we define $Z_n$ by $Z_{0,n} = n$ and ($\delta_k$ denotes the Dirac mass at $k \in \N$):
\[ {\offdistr}_{i,n} = \begin{cases}
	{\offdistr} & \text{ if } 0 < i < \Gamma_n,\\
	\delta_1 & \text{ if } \Gamma_n \leq i < 2 \Gamma_n,\\
	(1-p_n) \delta_0 + p_n \delta_1 & \text{ if } i = 2 \Gamma_n,\\
	{\offdistr} & \text{ if } i > 2 \Gamma_n
\end{cases} \]
for some vanishing sequence $p_n \in [0,1]$. Defining $\gamma_n(t) = \lfloor \Gamma_n t \rfloor$ and $X_n(t) = Z_{\gamma_n(t), n}/n$, we see that $X_n$ coincides (in distribution) with $Y_n$ on $[0,1)$, stays constant on $[1,2)$, undergoes a catastrophe (i.e., a highly subcritical offspring distribution, with mean $p_n$) at time $2$, and then resumes evolving according to ${\offdistr}$ after time $2$. The catastrophe at time $2$ is meant to correspond to a bottleneck, and indeed one can check that $\bottleneck(t) = 0$ if $t < 2$ and $\bottleneck(t) = 2$ if $t \geq 2$. Moreover, since $X_n$ shifted at time $2$ is a rescaled Galton Watson process, the discussion on Galton Watson processes in the next section will show that Assumption~\ref{assumptions} is satisfied.

Let us now see on this example various things that can go wrong around the bottleneck (here at time $2$), although Assumption~\ref{assumptions}
is satisfied. Fix some $t > 2$ and consider $u_n(s,t,\lambda) \to 0$ for $s < 2$. First of all, $u_n(s,t,\lambda)$ may not converge to $0$ and the uniqueness of \eqref{eq:dynamics-u} fails at the left of the bottleneck. Indeed, since $\widehat Y_n$ converges weakly to $\widehat Y$ and $\widehat Y$ is not conservative, there exist $\rho > 0$ and a sequence $y_n \to +\infty$ such that $\P(\widehat Y_n(2-) \geq y_n) \geq \rho$. In particular, just before the catastrophe $X_n$ is, with probability at least $\rho$, at least of the order of $y_n$. In this event, the catastrophe brings $X_n$ to level $p_n y_n$ by the law of large numbers, which diverges if $p_n \gg 1/y_n$, i.e., if $p_n$ vanishes slowly enough. This argument could be made rigorous to show that $u_n(s,t,\lambda)$ does not go to $0$ for $s < 1$. Secondly, even if $s < \bottleneck(t)$ the limit of $u_n(s,t,\lambda)$ depends on $s$: for $s < 1$ we have seen that the limit was $> 0$, while for $1 \leq s < 2$ the limit is $= 0$. Finally, $u_n(s,t,\lambda)$ may even fail to converge: to see this, one may for instance consider two sequences $p_n^{(1)}$ and $p_n^{(2)}$ with $y_n p_n^{(1)} \to +\infty$ and $y_n p_n^{(2)} \to 0$, $X_n^{(1)}$ and $X_n^{(2)}$ the two processes obtained by the above construction using $p_n^{(1)}$ and $p_n^{(2)}$ instead of $p_n$, respectively, and finally intertwine them by considering $X_{2n} = X_n^{(1)}$ and $X_{2n+1} = X_n^{(2)}$.
\\

This example therefore shows that a wide variety of behavior can happen before the bottleneck. We now give a sufficient condition that ensures $\bottleneck(t) = 0$, i.e., that there is no bottleneck before time $t$. Note that under this condition, Theorem~\ref{thm:main} then describes the behavior of $X_n$ on $[0,t]$. Intuitively, the following assumption ensures that $\xi_{i,n}$ is not too close to $0$, which avoids the almost sure absorption in one generation. For instance, it prevents the catastrophe of the previous example at time $2$.

\begin{Prop}[No bottleneck] \label{prop:no-bottleneck} 
	Let $t>0$. If for every $C > 0$
	\begin{equation} \label{eq:no-bottleneck}
		\liminf_{n \to +\infty} \left( \inf_{0 \leq i \leq \gamma_n(t)} \E\left(\xi_{i,n} ; \xi_{i,n} \leq Cn \right) \right) > 0,
	\end{equation}
	then $\bottleneck(t) = 0$.
\end{Prop}

We now conclude this discussion by giving a condition under which $u_n(s,t,\lambda) \to 0$ along a subsequence, uniformly in $s < \bottleneck(t)$ and $\lambda \geq 0$. Then the process started before the bottleneck  goes  as expected to
zero (along a subsequence) when going through the bottleneck. Note that the example given at the beginning of this section shows that this is not always the case. Roughly speaking, this condition means that the limiting process is conservative. Because of our definition of the bottleneck via a $\liminf$ in $n$, this is a challenging result to get ride of the subsequence.

\begin{Prop}[No explosion] \label{prop:no-explosion}
	Fix some $t>0$. If the two sequences $(\tv{\alpha_n}(t), n \geq 1)$ and $(\beta_n(t), n \geq 1)$ are bounded and
	\begin{equation} \label{eq:tightness}
		\lim_{A \rightarrow\infty} \ \sup_{n \geq 1, \, 0\leq s \leq y \leq t} \P(X_n(y) \geq A \ \vert \ X_n(s)=1) = 0, 
	\end{equation}
	then there exists an increasing sequence of integers $n(k)$ such that $u_{n(k)}(s,t,\lambda) \to 0$ as $k \to +\infty$, for all $s < \bottleneck(t)$ and $\lambda\geq 0$.
	
	Moreover, these assumptions are satisfied, i.e., $(\tv{\alpha_n}(t), n \geq 1)$ and $(\beta_n(t), n \geq 1)$ are bounded and~\eqref{eq:tightness} holds, if the following first moment condition is satisfied:
	\begin{equation} \label{eq:condition-first-moment}
		\sup_{n \geq 1} \left( n \sum_{i=0}^{\gamma_n(t)-1} \E\big(\lvert \overline \xi_{i,n}\rvert\big) \right) < +\infty.
	\end{equation}
\end{Prop}

\subsection{Comparison with earlier work} \label{sub:comparison}

In the Galton Watson case where ${\offdistr}_{i,n} = {\offdistr}_{0,n}$ and $\gamma_n(t) = \lfloor \Gamma_n t \rfloor$ for some sequence $\Gamma_n \to +\infty$, necessary and sufficient conditions for the finite-dimensional convergence of $(X_n, n \geq 1)$ are known since Grimvall~\cite{Grimvall74:0}. In this case, the next result shows that our Assumption~\ref{assumptions} is sharp.

\begin{Lem} \label{lemma:equivalence-GW}
	In the Galton Watson case, i.e., $\gamma_n(t) = \lfloor \Gamma_n t \rfloor$ and ${\offdistr}_{i,n} = {\offdistr}_{0,n}$, the sequence $(X_n, n \geq 1)$ converges in the sense of finite-dimensional distributions if and only if Assumption~\ref{assumptions} holds.
\end{Lem}

In the case of Galton Watson processes in varying environment, the first results seem to have been proved by Kurtz~\cite{Kurtz78:0} (see also Keiding~\cite{Keiding75:0} and Helland~\cite{Helland81:0} for the case of i.i.d.\ environment). Kurtz~\cite{Kurtz78:0} used semigroup techniques to study the case where offspring distributions have uniformly bounded third moments, which was later weakened by Borovkov~\cite{Borovkov02:0} to a second moment condition. There are three main differences between the assumptions made in~\cite{Borovkov02:0, Kurtz78:0} and our Assumption~\ref{assumptions}.

First, we do not need to assume uniformly bounded second moments. To our knowledge, this is the first result in that context and, as discussed above, this leads to new bottleneck phenomena. This also makes it possible to understand some subtle questions related to time scales that do not play a role in the finite variance case, see the discussion on Galton Watson processes in random environment in Section~\ref{subsub:GWRE}.

Second, as we already mentioned in the introduction, the function that in~\cite{Borovkov02:0, Kurtz78:0} essentially plays the role of our $\alpha_n$ is not assumed to have finite variations in~\cite{Borovkov02:0, Kurtz78:0}. It is important to note, however, that we use fundamentally different techniques that make it possible to go beyond the finite variance case: in particular, we end up with scaling limits outside the family obtained by Kurtz and Borovkov. Note also that this finite variation assumption is natural in our approach: otherwise it is not clear what meaning should be given to the term $\int_{(s,t]} u_{t, \lambda}(y) \alpha(dy)$ in~\eqref{eq:dynamics-u}. An enticing approach would be to consider $\alpha$ with finite quadratic variations, which would for instance make it possible to use a pathwise construction of It\^o's integral such as in F\"ollmer~\cite{Follmer81:0}, see also Wong and Zakai~\cite{Wong65:0}.

Finally, the functions that in~\cite{Borovkov02:0, Kurtz78:0} essentially play the role of our $\alpha_n$ and $\beta_n$ are assumed in~\cite{Borovkov02:0, Kurtz78:0} to converge in the $J_1$ topology, whereas here we only assume pointwise convergence (cf.\ Remark~\ref{remark:skorohod}).
\\

Let us finally mention an interesting and potentially fruitful connection with the convergence of processes with independent increments. In the proof of Lemma~\ref{lemma:equivalence-GW} (cf.\ Appendix~\ref{appendix:proof-GW}) we will prove that in the Galton Watson case, Assumption~\ref{assumptions} is equivalent to the convergence of some arrays of rowwise i.i.d.\ random variables. Actually, a more general result holds: indeed, starting from Theorem~VII.$4.4$ in~\cite{Jacod03:0}, it can be seen that Assumption~\ref{assumptions} is almost equivalent to the convergence of the sum
\begin{equation} \label{eq:triangular-array}
	\sum_{i = 1}^{\gamma_n(t)-1} \sum_{k=1}^n \overline \xi_{i,n}(k)
\end{equation}
where the $\overline \xi_{i,n}(k)$'s are independent and $\overline \xi_{i,n}(k)$ is distributed according to $\overline \xi_{i,n}$.

In the case of constant environment, this relation between triangular arrays and branching processes has a simple explanation: indeed, the Lamperti transformation transforms a branching process into a random walk and thus lies in-between these two objects. However, the Lamperti transformation breaks down for time-inhomogeneous Markov processes: the pre-image of a branching process (the so-called Lukasiewicz path, or breadth-first exploration) is no longer a random walk, and does not seem to have a simple probabilistic structure. It was therefore utterly surprising to us to end up with a condition which suggests that some process with independent increments plays a key role.

\subsection{Applications} \label{sub:applications}

We discuss in this section new results that stem from of Theorem~\ref{thm:main}. We keep the discussion at a high level and reserve rigorous results for future work (with the exception of Proposition~\ref{prop:Feller}).

\subsubsection{Scaling limits of Galton Watson processes in random (i.i.d.) environment} \label{subsub:GWRE}

Consider the case where for each $n \geq 1$, the sequence $({\offdistr}_{i,n}, i \geq 0)$ is i.i.d., distributed according to a random offspring distribution $Q_n$. Then the sequence $((\alpha_{i,n}, \beta_{i,n}, \nu_{i,n}), i \geq 0)$ is an i.i.d.\ sequence of $\R \times [0,\infty) \times \Mcal$-valued random variables, with $\Mcal$ the space of measures on $\R$. Because of the law of large numbers, it is therefore natural to choose $\gamma_n$ linear in $t$, i.e., $\gamma_n(t) = \lfloor \Gamma_n t \rfloor$ for some sequence $\Gamma_n \to +\infty$. We now discuss conditions under which Assumption~\ref{assumptions} holds.
\\

We are interested in the convergence of the process $Y_n(t) = (\alpha_n(t), \beta_n(t), M_n(t))$, where $M_n(t) = \sum_{0 \leq i < \gamma_n(t)} \nu_{i,n}$ defines a measure-valued process. The process $Y_n$ has i.i.d.\ increments, and so leveraging classical results on the convergence of measure-valued processes and on the convergence of random walks, we can get an explicit condition for its convergence. A function $h: \R^3 \to \R^3$ is called truncation function if it is continuous, bounded and satisfies $h(x) = x$ in a neighborhood of $0$. Let $\Ccal$ be a set of functions dense in the set of bounded, continuous functions, and for $\varphi \in \Ccal$ let $y_n^\varphi = (\alpha_{0,n}, \beta_{0,n}, \int \varphi d \nu_{0,n} )$.

\begin{condition}
	There exist a truncation function $h$, $F^\varphi$ a measure on $\R^3$ integrating $1 \wedge \lvert x\rvert^2$ and scalars $b^\varphi \in \R^3$, $c^\varphi_{ij} \geq 0$ such that for every $\varphi \in \Ccal$,
	\begin{multline} \label{eq:condition}
		\Gamma_n \E(h(y_n^\varphi)) \mathop{\longrightarrow}_{n \to +\infty} b^\varphi, \ \Gamma_n \left\{ \E\big[h_i(y_n^\varphi) h_j(y_n^\varphi)\big] - \E(h_i(y_n^\varphi)) \E(h_j(y_n^\varphi)) \right\} \mathop{\longrightarrow}_{n \to +\infty} c^\varphi_{ij}\\
		\text{and } \ \Gamma_n \E(g(y_n^\varphi)) \mathop{\longrightarrow}_{n \to +\infty} \int g(x) F^\varphi(dx).
	\end{multline}
	
	In the above, the second convergence holds for all $i,j = 1,2,3$ and the last convergence holds for all bounded, continuous functions $g$ that are equal to $0$ in a neighborhood of $0$.
\end{condition}

Assuming that this condition holds, it can be proved\footnote{Further details can be found in an earlier version of this paper at http://arxiv.org/pdf/1112.2547v3.pdf.} that $Y_n$ converges to the process $Y(t) = (\alpha(t), \beta(t), M(t))$ such that for every $\varphi$ continuous and bounded, the process $Y^\varphi = (\alpha, \beta, M^\varphi)$ with $M^\varphi = (\int \varphi(x) M(t)(dx), t \geq 0)$ is the L\'evy process with L\'evy exponent
\[ \psi^\varphi(v) = i v b^\varphi - \frac{1}{2} v c^\varphi v + \int \left( e^{i v x} - 1 - i v h(x) \right) F^\varphi(dx), \ v \in \R^3. \]

Further, using Skorohod's embedding theorem, we can assume that the convergence $Y_n \to Y$ holds almost surely. By definition, there exists a sequence of increasing bijections $(\lambda_n, n \geq 1)$ from $[0,\infty)$ to $[0,\infty)$ such that $\sup_{0 \leq s \leq t} \lvert \lambda_n(s) - s \rvert \to 0$ for every $t \geq 0$, and such that assumptions~\refAone\ (except for the convergence of $\tv{\alpha_n}$) and \refAtwo\ are satisfied for $\gamma'_n = \gamma_n \circ \lambda_n$ (see, e.g., Proposition VI.$2$.$1$ in Jacod and Shiryaev~\cite{Jacod03:0}).

Assuming now that $\alpha$ is of finite variations, $\alpha$ being a L\'evy process must be of the form $\alpha(t) = \d_\alpha t + S_+(t) - S_-(t)$ where $\d_\alpha \in \R$ and $S_+$ and $S_-$ are two independent pure-jump subordinators (see, e.g., Bertoin~\cite{Bertoin96:0}). With this special structure, it is possible to prove that $\tv{\alpha_n}(t) \to \tv{\alpha}(t)$ so that Assumption~\ref{assumptions} is fully satisfied and all the conclusions of Theorem~\ref{thm:main} hold. It would be interesting to delve deeper into the probabilistic structure of the process $(\alpha, \beta, \nu)$, and to understand how it relates to the properties of the limiting process $X$ such as the extinction probability or the speed of extinction. In the literature, only the case of Feller diffusion in random environment where $\nu = 0$ and $\alpha$ is a Brownian motion has begun to be looked at, see, e.g., B\"oinghoff and Hutzenthaler~\cite{Boinghoff12:0}.
\\

We conclude this section by commenting on a question that actually motivated us in the first place: given a sequence of Galton Watson processes in random environment, how can we find the right renormalization in time, i.e., the right sequence $(\Gamma_n)$?

For the sake of the discussion, we consider one of the simplest possible case where in each generation we choose at random among one of two possible offspring distributions, i.e., we can write $Q_n = p_n^{(1)} \delta_{q^{(1)}} + p_n^{(2)} \delta_{q^{(2)}}$ where $p_n^{(j)} \in [0,1]$, $p_n^{(1)} + p_n^{(2)} = 1$ and $q^{(1)}$, $q^{(2)}$ are two offspring distributions. In this discussion, we will call a CSBP with characteristic $(b,c,F)$ the CSBP whose branching mechanism is given by $\psi(\lambda) = \lambda b - \frac{1}{2} c \lambda^2 + \int(e^{-\lambda x} - 1 - \lambda x \indicator{x \leq 1}) F(dx)$. For each $j = 1,2$ let $Z_n^{(j)} = (Z_n^{(j)}(i), i \geq 0)$ be a Galton Watson process with offspring distribution $q^{(j)}$ and consider $(\Gamma^{(j)}_n)$ a sequence such that $(X_n^{(j)}, n \geq 1)$ converges weakly to the CSBP with characteristic $(b^{(j)}, c^{(j)}, F^{(j)})$, where $X_n^{(j)}(t) = n^{-1} Z_n^{(j)}(\lfloor \Gamma_n^{(j)}t \rfloor)$.

If both $q^{(1)}$ and $q^{(2)}$ have finite variance, then it is well-known that in order to renormalize the Galton Watson process with offspring distribution $q^{(i)}$ when the space scale is $n$, one needs to speed up time with $n$ also, i.e., $\Gamma^{(1)}_n = \Gamma^{(2)}_n = n$. Thus when ``mixing'' these two processes, it is natural to speed up the resulting process by the common time scale and thus take $\Gamma_n = n$. To our knowledge, only such cases have been considered in the literature so far. When offspring distributions have infinite variance however, the situation becomes more delicate. Indeed, if for instance $q^{(1)}([x,\infty)) \sim x^{-a}$ as $x \to +\infty$ for some $a \in (1,2)$, then one needs to consider $\Gamma^{(1)}_n = n^{a-1}$. Thus there are now two ``natural'' time scales, namely $\Gamma^{(1)}_n = n^{a-1}$ and $\Gamma^{(2)}_n = n$.

To understand how one should choose $\Gamma_n$ and what should be the limit, it is useful to have the following interpretation in mind: $\Gamma_n^{(j)}$ is the number of generations needed in order for $Z_n^{(j)}$ to evolve by $n$. Consider now the Galton Watson process in random environment that mixes $q^{(1)}$ and $q^{(2)}$ via $Q_n$ as above: then over $\Gamma_n$ generations, the law of large numbers implies that $q^{(j)}$ has been used $p_n^{(j)} \Gamma_n$ times. Thus, if $p_n^{(j)} \Gamma_n \ll \Gamma_n^{(j)}$, the offspring distribution $q^{(j)}$ has not been picked sufficiently often in order to have any effect (on the space scale $n$). This suggests that the correct time scale is $\Gamma_n = \min_j (\Gamma_n^{(j)}/p_n^{(j)})$ and indeed, the following result can be proved using Theorem~\ref{thm:main}:
\begin{itemize}
	\item if $\Gamma_n^{(1)} / p_n^{(1)} \ll \Gamma_n^{(2)}/p_n^{(2)}$ and $\Gamma_n = \Gamma_n^{(1)}/p_n^{(1)}$, then $X_n$ converges toward the CSBP with branching mechanism $(b^{(1)}, c^{(1)}, F^{(1)})$;
	\item if $\Gamma_n^{(2)} / p_n^{(2)} \ll \Gamma_n^{(1)}/p_n^{(1)}$ and $\Gamma_n = \Gamma_n^{(2)}/p_n^{(2)}$, then $X_n$ converges toward the CSBP with branching mechanism $(b^{(2)}, c^{(2)}, F^{(2)})$;
	\item if $\Gamma_n^{(1)} p_n^{(2)} / (\Gamma_n^{(2)} p_n^{(1)}) \to \ell \in (0,\infty)$ and $\Gamma_n = \Gamma_n^{(1)}/p_n^{(1)}$, then $X_n$ converges toward the CSBP with characteristic $(b^{(1)} + \ell b^{(2)}, c^{(1)} + \ell c^{(2)}, F^{(1)} + \ell F^{(2)})$.
\end{itemize}

This discussion can be easily extended to the case of a finite number of offspring distributions that also vary with $n$, and it would be very interesting to understand the implications of Theorem~\ref{thm:main} in more general settings, e.g., when we can choose among uncountably many offspring distributions.

\subsubsection{Feller diffusion} Going back to the case of varying environment, the finite variance case is of particular interest. This is the only one that has been studied so far, see in particular~\cite{Borovkov02:0, Kurtz78:0}. In this case, our approach via the generalized branching equation~\eqref{eq:dynamics-u} makes it possible to derive a necessary and sufficient condition for the extinction probability of a Feller diffusion in varying environment, see~\eqref{eq:proba-extinction} below. This extends results already known for linear birth and death branching processes in varying environment from~\cite{Kendall48:0} and for particular classes of CSBP in random environment from~\cite{BPMS, Boinghoff12:0}.

\begin{Prop} \label{prop:Feller}
	Assume that Assumption~\ref{assumptions} holds with $\nu = 0$. Then $\beta$ is continuous and for all $t\geq 0$, we have
	\begin{equation} \label{eq:solution-Feller}
		u(s,t,\lambda) = \frac{\exp(-\overline \alpha(s))}{\lambda^{-1} \exp(-\overline\alpha(t)) + \int_{(s,t]} \exp(-\overline\alpha(y)) \beta(dy)}, \ 0 \leq s \leq t, \lambda \geq 0,
	\end{equation}
	where $\overline{\alpha}(t)=\alpha (t)+\sum_{0 \leq s \leq t} [\log(1+ \Delta \alpha(s))-\Delta \alpha(s)]$. In particular, if $\bottleneck(t) = 0$ for every $t \geq 0$ (for instance, if~\eqref{eq:no-bottleneck} holds), then for any $s \geq 0$ and $x \geq 0$
	\begin{equation} \label{eq:proba-extinction}
		\lim_{t \to +\infty} \P(X(t) = 0 \mid X(s) = x) = \exp \left(-\frac{x\exp(-\overline \alpha(s))}{\int_{(s,\infty)} \exp(-\overline \alpha(y)) \beta(dy) }\right)
	\end{equation}
	where $X$ is the weak limit of the sequence of processes $(X_n(y), y \geq s)$ given by properties~\ref{property:fdd-convergence} and~\ref{property:tightness} of Theorem~\ref{thm:main}.
\end{Prop}

\begin{proof}
	Since Assumption~\ref{assumptions} holds, all the conclusions of Theorem~\ref{thm:main} hold. In particular, $\widetilde \beta$ is continuous and since $\nu = 0$ by assumption, $\beta = \widetilde \beta$ and $\beta$ itself is continuous.
	
	Let us now prove~\eqref{eq:solution-Feller}. Fix $t, \lambda > 0$: according to Theorem~\ref{thm:main}, it is enough to check that $G(s) = H(s)$, where $G(s)$ is equal to the right-hand side of~\eqref{eq:solution-Feller} and $H(s) = \lambda + \int_{(s,t]} G d\alpha - \int_{(s,t]} G^2 d\beta$. Observe that $G$ and $H$ may only jump when $\alpha$ does. We first compare the jumps: since $\beta$ is continuous, $s \mapsto \int_{(s,t]} \exp(-\overline \alpha(y)) d\beta (y)$ is continuous and so
	\[ \Delta G(s) = \frac{\exp(-\overline \alpha(s)) - \exp(-\overline \alpha(s-))}{ \lambda^{-1} \exp(-\overline\alpha(t)) + \int_{(s,t]} \exp(-\overline\alpha(y)) \beta(dy)} = G(s) \left( 1 - e^{\Delta \overline \alpha(s)} \right). \]
	
	Since by definition $\Delta \overline \alpha(s) = \log(1 + \Delta \alpha(s))$ we obtain $\Delta G(s) = -G(s) \Delta \alpha(s)$ which coincides with $\Delta H(s)$ (since $s \mapsto \int_{(s,t]} G^2 d\beta$  is continuous). Let us now compare $G$ and $H$ outside the jumps of $\alpha$, so that $\overline \alpha(ds) = \alpha(ds)$: starting from the right-hand side of~\eqref{eq:solution-Feller}, the chain rule for functions of bounded variations gives
	\begin{multline*}
		dG(s) = \frac{- \alpha(ds) \exp(-\overline \alpha(s))}{\lambda^{-1} \exp(-\overline\alpha(t)) + \int_{(s,t]} \exp(-\overline\alpha(y)) \beta(dy)}\\
		+ \frac{\exp(-2\overline \alpha(s)) \beta(ds)}{\left(\lambda^{-1} \exp(-\overline\alpha(t)) + \int_{(s,t]} \exp(-\overline\alpha(y)) \beta(dy)\right)^2},
	\end{multline*}
	i.e., $dG(s) = - G(s) \alpha(ds) + G(s)^2 \beta(ds) = dH(s)$. This proves~\eqref{eq:solution-Feller} from which~\eqref{eq:proba-extinction} follows from the facts that $\P(X(t) = 0 \mid X(s) = x) = \lim_{\lambda \to +\infty} \E(e^{-\lambda X(t)} \mid X(s) = x)$ and that $\E(e^{-\lambda X(t)} \mid X(s) = x) = \exp(-xu(s,t,\lambda))$.
\end{proof}

\subsubsection{Remarks on CSBP with catastrophes}

Theorem~\ref{thm:main} makes it possible to study Galton Watson processes where only few offspring distributions are not near-critical. The simplest example is given by taking $\gamma_n(t) = \lfloor \Gamma_n t \rfloor$ and ${\offdistr}_{i,n} = {\offdistr}_{0,n}$, in such a way that the corresponding sequence of renormalized Galton Watson processes converges to a CSBP. Then, for some $t_0 \geq 0$, one can change ${\offdistr}_{\gamma_n(t_0), n}$ and take its mean equal to $1 + a$. Then $(X_n)$ converges to a process $X$ which is a CSBP on $[0,t_0)$ and on $[t_0,\infty)$ and such that $X(t_0) = (1+a) X(t_0-)$. Such processes with catastrophes have been studied in~\cite{BPMS} with motivations for cell division models. More precisely, CSBP's are then multiplied by some random number described by a Poisson point process, whose associated L\'evy process has finite variations. Theorem~\ref{thm:main} thus yields an alternative way to construct the process and characterize its Laplace exponent (whereas~\cite{BPMS} leverages results on stochastic differential equations with jumps).

Another way to create a discontinuity at a fixed time is to take ${\offdistr}_{\gamma_n(t_0), n} = (1-1/n)\delta_0+(1/n)\delta_n$ as in the example considered in the beginning of Section~\ref{sub:bottleneck}. Again, $(X_n)$ converges to a process $X$ which is a CSBP on $[0,t_0)$ and on $[t_0,\infty)$ and such that $X(t_0) = S(X(t_0-))$ with $(S(x), x \geq 0)$ a Poisson process. Theorem~\ref{thm:main} allows accumulation of such fixed jumps; note that in both cases these jumps may be negative, whereas CSBP's only have positive jumps.

Building on these two simple examples, we expect in general that if $X$ is a time-inhomogeneous Markov process satisfying the branching property, then for each fixed time of discontinuity $t$, there should exist a subordinator $S_t = (S_t(x), x \geq 0)$ such that $X(t) = S_t(X(t-))$. Indeed, preliminary results suggest that the Markov property should imply the existence of such a process $S_t$, while the branching property of $X$ would force $S_t$ to be a subordinator.

\section{Additional notation and preliminary results} \label{sec:additional-notation-preliminary-results}

In this section we gather some notation used throughout the rest of the paper. Of particular importance are the constants defined in Section~\ref{sub:constants}, which will be used repeatedly in the proofs.

\subsection{Additional notation} \label{sub:additional-notation}

First, note that $\alpha_{i,n}$ and $\beta_{i,n}$ can be rewritten in terms of $\nu_{i,n}$ as follows:
\begin{equation} \label{eq:alternative-def}
	\alpha_{i,n} = \int \frac{x}{1+x^2} \nu_{i,n}(dx) \ \text{ and } \ \beta_{i,n} = \frac{1}{2} \int \frac{x^2}{1+x^2} \nu_{i,n}(dx).
\end{equation}

From now on we identify any c\`adl\`ag function of locally finite variation $f$ with its corresponding signed measure, see for instance Chapter~$3$ in Kallenberg~\cite{Kallenberg02:0}. For instance, we will write indifferently $f((s,t])$, $f(s,t]$ or $f(t) - f(s)$ for $0 \leq s \leq t$, as well as $\Delta f(t)$ or $f\{t\}$. Let $g$ and $h$ be defined as follows:
\begin{equation} \label{eq:def-g-h}
	g(x, \lambda) = 1 - e^{-\lambda x} - \frac{\lambda x}{1+x^2} \ \text{ and } \ h(x,\lambda) = g(x,\lambda) + \frac{(\lambda x)^2}{2(1+x^2)}, \ x \in \R, \lambda \geq 0.
\end{equation}

Defining for $x \in \R$
\begin{equation} \label{eq:def-Phi}
	\Phi_1(x) = \frac{e^{-x}-1+x}{x^2} \ \text{ and } \ \Phi_2(x) = \frac{-e^{-x}+1-x+x^2/2}{x^2},
\end{equation}
with $\Phi_1(0) = 1/2$ and $\Phi_2(0) = 0$, it will sometimes be convenient to write
\begin{equation} \label{eq:g-h-Phi}
	g(x,\lambda) = \frac{x^2}{1+x^2} \left(1-e^{-\lambda x} - \lambda^2 \Phi_1(\lambda x)\right) \text{ and } h(x,\lambda) = \frac{x^2}{1+x^2} \left(1-e^{-\lambda x} + \lambda^2 \Phi_2(\lambda x)\right).
\end{equation}

\medskip

For $n \geq 1$ let in the sequel $\mu_n = \tv{\alpha_n} + \beta_n$, i.e.,
\begin{equation} \label{eq:def-mu-n}
	\mu_n(t) = \tv{\alpha_n}(t) + \beta_n(t), \ t \geq 0,
\end{equation}
and for $i \geq 0$ and $n \geq 1$, let
\begin{equation} \label{eq:def-psi-in}
	\psi_{i,n}(\lambda) = u_n(t^n_i, t^n_{i+1}, \lambda) - \lambda = -n\log \left(1 - \frac{1}{n} \int \left(1 - e^{-\lambda x} \right) \nu_{i,n}(d x) \right), \ \lambda \geq 0
\end{equation}
(the second equality is derived after some algebra by starting from the definition of $u_n$ as Laplace exponent of $X_n$). In order to use the approximation $\psi_{i,n}(\lambda) \approx \int(1-e^{-\lambda x}) \nu_{i,n} (dx)$, we introduce the function $\epsilon_{i,n}$ such that
\begin{equation} \label{eq:def-epsilon}
	\psi_{i,n}(\lambda) = \left( 1+\epsilon_{i,n}(\lambda) \right) \int \left( 1 - e^{-\lambda x} \right) \nu_{i,n}(d x),
\end{equation}
with $\epsilon_{i,n}(\lambda) = 0$ when $\int(1-e^{-\lambda x}) \nu_{i,n}(dx) = 0$.
\\

For every $n \geq 1$ and every measurable, positive function $f: [0,\infty) \to (0,\infty)$, define the two measures $\Psi(f)$ and $\Psi_n(f)$ as follows:
\[ \Psi(f)(A) = \int_{A} f(y) \alpha(dy) - \int_{A} f(y)^2 \beta(dy) + \int_{(0,\infty) \times A} h(x, f(y)) \nu (dx \, dy), \ A \in \Bcal, \]
with $\Bcal$ the Borel subsets of $\R$, and
\[ \Psi_n(f)(A) = \sum_{i \geq 1} \indicator{t_{i}^n \in A} \psi_{i-1,n}(f(t_{i}^n)), \ A \in \Bcal. \]

With a slight abuse of notation, we will also consider $\Psi(f)$ and $\Psi_n(f)$ for functions $f$ only defined on a subset of $[0,\infty)$, typically $[\bottleneck(t), t]$. Then we will only consider $\Psi(f)(A)$ or $\Psi_n(f)(A)$ for Borel sets $A$ which are subset of the domain of definition of $f$.

\subsection{Heuristic derivation of~\eqref{eq:dynamics-u}}

The rational for introducing the measure-valued operators~$\Psi_n$ and~$\Psi$ is the following. On the one hand, it follows readily from the various definitions made that~\eqref{eq:dynamics-u} can be rewritten as
\begin{equation} \label{eq:dynamics-u-Psi}
	u(s, t, \lambda) = \lambda + \Psi(u(\, \cdot \, , t, \lambda))((s,t]).
\end{equation}

On the other hand, $\Psi_n$ has been chosen so that $u_n$ satisfies a similar dynamics. To see this, note that from the definition of $u_n$ and the Markov property of $X_n$, we get the following composition rule:
\begin{equation} \label{eq:composition-rule}
	u_n(t_1, t_3, \lambda) = u_n(t_1, t_2, u_n(t_2, t_3, \lambda)), \quad 0 \leq t_1 \leq t_2 \leq t_3, \ \lambda \geq 0.
\end{equation}

\begin{Lem}\label{relun}
	For any $n \geq 1$, $\lambda \geq 0$ and $0 \leq s \leq t$, it holds that
	\begin{equation} \label{eq:dynamics-u-n}
		u_n(s,t,\lambda) = \lambda + \sum_{i=\gamma_n(s)+1}^{ \gamma_n(t)} \psi_{i-1,n}(u_n(t_{i}^n,t,\lambda)) = \lambda + \Psi_n(u_n(\, \cdot \, , t, \lambda))((s,t]).
	\end{equation}
\end{Lem}

\begin{proof}
	The second equality follows readily from the definition of $\Psi_n$, while the first one can be derived as follows:
	\begin{align*}
		u_n(s,t,\lambda) = u_n(t^n_{\gamma_n(s)}, t, \lambda) & = \lambda + \sum_{i=\gamma_n(s)}^{\gamma_n(t)-1} \left( u_n(t^n_i, t, \lambda) - u_n(t^n_{i+1}, t, \lambda) \right)\\
		& \stackrel{\text{(i)}}{=} \lambda + \sum_{i=\gamma_n(s)}^{\gamma_n(t)-1} \left( u_n(t^n_i, t^n_{i+1}, u_n(t^n_{i+1}, t, \lambda)) - u_n(t^n_{i+1}, t, \lambda) \right)\\
		& \stackrel{\text{(ii)}}{=} \lambda + \sum_{i=\gamma_n(s)}^{\gamma_n(t)-1} \psi_{i,n}(u_n(t^n_{i+1}, t, \lambda))
	\end{align*}
	where~(i) comes from the composition rule~\eqref{eq:composition-rule} and~(ii) comes from the first equality in~\eqref{eq:def-psi-in}.
\end{proof}

From~\eqref{eq:dynamics-u-n} we can intuitively recover~\eqref{eq:dynamics-u-Psi} (i.e.,~\eqref{eq:dynamics-u}). Indeed, in view of the second equality in~\eqref{eq:def-psi-in} and of the approximation $\log(1-x) \approx -x$, it is reasonable to expect
\[ \psi_{i,n}(\lambda) \approx \int \left(1 - e^{-\lambda x} \right) \nu_{i,n}(d x) = \lambda \alpha_{i,n} - \lambda^2 \beta_{i,n} + \int_{(0,\infty)} h(x, \lambda) \nu_{i,n} (dx) \]
(recall~\eqref{eq:alternative-def} and the definition~\eqref{eq:def-g-h} of $h$ for the last equality) and so summing over $i = 0, \ldots, \gamma_n(t)-1$ yields through~\eqref{eq:dynamics-u-n} the approximation
\begin{multline*}
	u_n(s,t,\lambda) \approx \lambda + \int_{(s,t]} u_n(y, t, \lambda) \alpha_{n}(dy) - \int_{(s,t]} (u_n(y, t, \lambda))^2 \beta_{n}(dy)\\
	+ \int_{(0,\infty) \times (s,t]} h(x, u_n(y, t, \lambda)) \nu_{n} (dx \, dy).
\end{multline*}

Since $(\alpha_n, \beta_n, \nu_n)$ is assumed to converge toward $(\alpha, \beta, \nu)$, this last approximation suggests that any limit $u(s,t,\lambda)$ of the sequence $(u_n(s,t,\lambda))$ should indeed satisfy~\eqref{eq:dynamics-u-Psi}.

\subsection{Key constants} \label{sub:constants}

For any $n \geq 1$, $t, \lambda, C \geq 0$, $s \leq t$, $N \geq 1$, $0 < \eta < T$, let:
\begin{equation} \label{eq:c_1}
	c_1(C) = C + c'_1(C) \ \text{ with } \ c_1'(C) = \sup \left\{ \frac{2 \lvert g(x, \lambda)\rvert (1+x^2)}{x^2} : x \geq -1, 0 \leq \lambda \leq C \right\},
\end{equation}
\begin{equation} \label{eq:c2+c3}
	c_2(\eta, T) = \sup_{\substack{\eta \leq y,y' \leq T \\ 0 \leq x}} \left\lvert \frac{h(x,y) - h(x,y')}{(y-y') x^2 / (1+x^2)} \right\rvert \text{ and } c_3(\eta, T) = 1 + T + c_2(\eta, T),
\end{equation}
\begin{equation} \label{eq:def-c^epsilon}
	\overline c_{n,t}^\epsilon(C) = \sup \left \{ \vert \epsilon_{i,n}(\lambda) \vert : 0 \leq i < \gamma_n(t), 0 \leq \lambda \leq C \right\},
\end{equation}
\begin{equation} \label{eq:def-overline-c^u}
	\overline c_{t, \lambda}^u = \sup \left\{ u_n(s,t,\lambda) : n \geq 1, 0 \leq s \leq t \right\},
\end{equation}
\begin{equation} \label{eq:def-Delta}
	\Delta_{t,\lambda}^u = \left( 1 + \sup_{n \geq 1} \left\{ \overline c_{n,t}^\epsilon \big( \overline c_{t, \lambda}^u \big) \right\} \right) c_1\big(\overline c_{t,\lambda}^u\big),
\end{equation}
\begin{equation} \label{eq:def-underline-c^u}
	\underline c_{s, t, \lambda}^u(N) = \inf \left\{ u_n(y, t, \lambda): s \leq y \leq t, n \geq N \right\}
\end{equation}
and $N_{s,t,\lambda} = \inf\big\{ N \geq 1: \underline c_{s,t,\lambda}^u(N) > 0 \big\}$. When $N_{s,t,\lambda}$ is finite, we also define
\begin{equation} \label{eq:def-underline-c^u-2}
	\underline c^u_{s,t,\lambda} = \underline c^u_{s,t,\lambda}(N_{s,t,\lambda})
\end{equation}
in which case $\underline c_{s,t,\lambda}^u > 0$. We defer the proofs that these constants are finite to Appendix~\ref{appendix:constants} and now use these constants to prove key results. Of particular importance are Lemma~\ref{lemma:inequality-delta}, which controls fluctuations of $u_n(s,t,\lambda)$ in $s$, and Lemma~\ref{lemma:inf} which allows to rewrite the time of last bottleneck $\bottleneck(t)$ in a more convenient form.

\begin{Lem} \label{lemma:inequality-c1}
	For any $C \geq 0$, $n \geq 1$ and $i \geq 0$,
	\begin{equation} \label{eq:ineq-g-2}
		\sup_{0 \leq \lambda \leq C} \left\lvert \int \left( 1 - e^{-\lambda x} \right) \nu_{i,n}(dx) \right\rvert \leq c_1(C) \mu_n(t_{i}^n, t_{i+1}^n].
	\end{equation}
\end{Lem}

\begin{proof}
	By definition~\eqref{eq:def-g-h} of $g$, we have
	\[ \int \left( 1 - e^{-\lambda x} \right) \nu_{i,n}(dx) = \lambda \alpha_{i,n} + \int g(x, \lambda) \nu_{i,n}(dx) \]
	so that $\lvert\int(1 - e^{-\lambda x})\nu_{i,n}(dx)\rvert \leq \lambda \lvert\alpha_{i,n}\rvert + \int \lvert g(x, \lambda)\rvert \nu_{i,n}(dx)$. Since
	\[ \lvert g(x,\lambda)\rvert \leq c_1'(C) \frac{x^2}{2(1+x^2)} \]
	for all $x \geq -1$ and $0 \leq \lambda \leq C$ by definition of $c_1'(C)$, we get~\eqref{eq:ineq-g-2}.
\end{proof}

\begin{Lem} \label{lemma:inequality-delta}
	For any $n \geq 1$, $\lambda, t > 0$ and $0 \leq s \leq s' \leq t$,
	\begin{equation} \label{eq:variation}
		\left\lvert u_n(s,t,\lambda) - u_n(s',t,\lambda) \right\rvert \leq \Delta_{t,\lambda}^u \mu_n(s,s'].
	\end{equation}
\end{Lem}

\begin{proof}
	Lemma~\ref{relun} and the definition of $\epsilon_{i,n}$ give
	\begin{multline*}
		\left\lvert u_n(s,t,\lambda) - u_n(s',t,\lambda) \right\rvert\\
		\leq \sum_{i=\gamma_n(s)+1}^{\gamma_n(s')} \left( 1 + \left\lvert \epsilon_{i-1,n}(u_n(t_i^n, t, \lambda)) \right\rvert \right) \left\lvert \int (1-e^{-x u_n(t_i^n, t, \lambda)}) \nu_{i-1,n}(dx) \right\rvert.
	\end{multline*}

	Since $0 \leq t_i^n \leq t$ for any $0 \leq i \leq \gamma_n(t)$, we have $u_n(t_i^n, t, \lambda) \leq \overline c_{t,\lambda}^u$ and in particular $\lvert\epsilon_{i-1,n}(u_n(t_i^n, t, \lambda))\rvert \leq \overline c_{n,t}^\epsilon ( \overline c_{t, \lambda}^u )$ for all $\gamma_n(s) < i \leq \gamma_n(s')$. Using in addition~\eqref{eq:ineq-g-2} with $C = \overline c_{t,\lambda}^u$, we obtain
	\[
		\left\lvert u_n(s,t,\lambda) - u_n(s',t,\lambda) \right\rvert \leq \sum_{i = \gamma_n(s)+1}^{\gamma_n(s')} \left( 1 + \overline c_{n,t}^\epsilon ( \overline c_{t, \lambda}^u ) \right) c_1(\overline c_{t,\lambda}^u) \mu_n(t_{i-1}^n, t_{i}^n] = \Delta_{t, \lambda}^u \mu_n(s, s']
	\]
	which gives~\eqref{eq:variation}.
\end{proof}

\begin{Rk} \label{rk:relative-compactness}
	When $\Delta^u_{t, \lambda}$ is finite and $\mu_n \to \mu$ (in Skorohod topology),~\eqref{eq:variation} implies that the sequence of functions $(u_n(\, \cdot \,,t,\lambda), n \geq 1)$ on $[0,t]$ is relatively compact in view of the Arzel\`a-Ascoli theorem. These assumptions are satisfied when Assumption~\ref{assumptions} holds, so that the convergence $u_n(s,t,\lambda) \to u(s,t,\lambda)$ of Theorem~\ref{thm:main} actually holds in a functional sense. However, we will not make use of this result.
\end{Rk}

In the next lemma we provide an alternative expression for $\bottleneck(t)$, defined so far as $\sup \Scal(t)$ with
\[ \Scal(t) = \left\{s \leq t : \lim_{\varepsilon \to 0} \liminf_{n\rightarrow\infty} \inf_{s \leq y \leq t} \P\left(X_n(t) > \varepsilon \ \big \vert \ X_n(y)=1 \right) = 0 \right\}. \]

More precisely, we show that $\bottleneck(t) = \sup \Scal(t,\lambda)$ where
\[ \Scal(t, \lambda) = \left\{s \leq t : \liminf_{n \rightarrow \infty} \inf_{s \leq y \leq t} u_n(y,t,\lambda)=0\right\}. \]

In particular, we deduce that $u_n$ past $\bottleneck(t)$ is uniformly bounded away from $0$ (i.e., $N_{s,t,\lambda}$ is finite and $\underline c^u_{s,t,\lambda} > 0$ for $s > \bottleneck(t)$), which is in line with the intuition behind $\bottleneck(t)$ being the last bottleneck before time $t$.

\begin{Lem} \label{lemma:inf}
	Fix $t > 0$ and assume that the sequences $(\tv{\alpha_n}(t), n \geq 1)$ and $(\beta_n(t), n \geq 1)$ are bounded. Then $\bottleneck (t) = \sup \Scal(t, \lambda)$ for every $\lambda > 0$ and $N_{s,t,\lambda}$ is finite for every $s \in (\bottleneck(t), t]$.
	
	In particular, if $(\tv{\alpha_n}(t), n \geq 1)$ and $(\beta_n(t), n \geq 1)$ are bounded for every $t \geq 0$, then the function $t \mapsto \bottleneck(t)$ is increasing.
\end{Lem}

\begin{proof}
	Fix in the rest of the proof $t, \lambda > 0$ and let $s \leq t$: the following statements are equivalent, which proves that $\Scal(t) = \Scal(t, \lambda)$ and implies $\bottleneck(t) = \sup \Scal(t, \lambda)$:
	\begin{enumerate}[(i)]
		\item $\liminf_{n \rightarrow \infty} \inf_{s \leq y \leq t} u_n(y,t,\lambda)=0$;
		\item there exist sequences $(n(k))$ and $(y_k)$ such that $y_k \in [s,t]$ for each $k \geq 1$ and
		\[ \lim_{k \to +\infty} n(k) = +\infty \ \text{ and } \lim_{k \to +\infty} u_{n(k)}(y_k, t, \lambda) = 0; \]
		\item \label{item:cond-laplace} there exist sequences $(n(k))$ and $(y_k)$ such that $y_k \in [s,t]$ for each $k \geq 1$ and
		\[ \lim_{k \to +\infty} n(k) = +\infty \ \text{ and for every } v > 0, \lim_{k \to +\infty} \E\left(e^{-v X_{n(k)}(t)} \, \lvert \, X_{n(k)}(y_k) = 1\right) = 1; \]
		\item \label{item:cond-conv-distrib} there exist sequences $(n(k))$ and $(y_k)$ such that $y_k \in [s,t]$ for each $k \geq 1$ and for any $\varepsilon > 0$,
		\[ \lim_{k \to +\infty} n(k) = +\infty \text{ and } \lim_{k \to +\infty} \P \left( X_{n(k)}(t) > \varepsilon \, \rvert \, X_{n(k)}(y_k) = 1 \right) = 0; \]
		\item \label{item:cond-last} $\lim_{\varepsilon \to 0} \liminf_{n\rightarrow\infty} \inf_{s \leq y \leq t} \P\big(X_n(t) > \varepsilon \ \vert \ X_n(y) = 1 \big)=0.$
	\end{enumerate}
	
	The equivalence between~\eqref{item:cond-laplace} and~\eqref{item:cond-conv-distrib} relies on the fact that both conditions are equivalent to the following one: the sequence of random variables $(X_{n(k)}(t), k \geq 1)$ under $\P(\, \cdot \mid X_{n(k)}(y_k) = 1)$ converges in distribution to $0$. Let us also explain the last equivalence. The condition~\eqref{item:cond-conv-distrib} implies that
	\[ \liminf_{n\rightarrow\infty} \inf_{s \leq y \leq t} \P\big(X_n(t) > \varepsilon \ \vert \ X_n(y) = 1 \big)=0 \]
	for every $\varepsilon > 0$, which is stronger than~\eqref{item:cond-last}. Now, assuming that~\eqref{item:cond-last} holds, one can find sequences $(\varepsilon_k)$, $(n(k))$ and $(y_k)$ such that $y_k \in [s,t]$ and
	\[ \lim_{k \to +\infty} \varepsilon_k = 0, \lim_{k \to +\infty} n(k) = +\infty \ \text{ and } \ \lim_{k \to +\infty} \P \left( X_{n(k)}(t) > \varepsilon_k \mid X_{n(k)}(y_k) = 1 \right) = 0. \]
	
	Then the sequences $(n(k))$ and $(y_k)$ satisfy~\eqref{item:cond-conv-distrib} since for any $\varepsilon > 0$,
	\[ \P \left( X_{n(k)}(t) > \varepsilon \mid X_{n(k)}(y_k) = 1 \right) \leq \P \left( X_{n(k)}(t) > \varepsilon_k \mid X_{n(k)}(y_k) = 1 \right) \]
	for $k$ large enough, since $\varepsilon_k \to 0$. This proves $\bottleneck(t) = \sup \Scal(t, \lambda)$, which implies that $N_{s,t,\lambda}$ is finite when $\bottleneck(t) < s \leq t$ since from the definition~\eqref{eq:def-underline-c^u} of $\underline c^u_{s,t,\lambda}(N)$,
	\[ \lim_{N \to +\infty} \underline c_{s,t,\lambda}^u(N) = \liminf_{n \rightarrow \infty} \inf_{s \leq y \leq t} u_n(y,t,\lambda). \]
	
	We now assume that $(\tv{\alpha_n}(t))$ and $(\beta_n(t))$ are bounded for every $t \geq 0$, and prove that $\bottleneck(\, \cdot \,)$ is an increasing function. Let $t' > t$: we will show that $\Scal(t, \overline c_{t',\lambda}^u) \subset \Scal(t', \lambda)$, which proves that $\bottleneck(t) \leq \bottleneck(t')$. So consider $s \in \Scal(t, \overline c_{t',\lambda}^u)$, i.e., $s \leq t$ with
	\[ \liminf_{n \to +\infty} \inf_{s \leq y \leq t} u_n(y,t,\overline c_{t',\lambda}^u) = 0. \]
	
	Then $s \leq t'$, and the composition rule~\eqref{eq:composition-rule} together with the monotonicity of $u_n$ in $\lambda$ give for any $s \leq y \leq t$
	\[ u_n(y,t',\lambda) = u_n(y,t,u_n(t,t',\lambda)) \leq u_n(y,t,\overline c^u_{t',\lambda}) \]
	which entails
	\[ \liminf_{n \to +\infty} \inf_{s \leq y \leq t'} u_n(y,t', \lambda) \leq \liminf_{n \to +\infty} \inf_{s \leq y \leq t} u_n(y,t,\overline c_{t',\lambda}^u). \]
	
	Since this last quantity is equal to $0$ this proves that $s \in \Scal(t', \lambda)$ and gives the result.
\end{proof}

\section{Proof of Theorem~\ref{thm:main}} \label{proofThm1}

In the rest of this section, we assume that Assumption~\ref{assumptions} holds and we consider the measures $\alpha$, $\beta$ and $\nu$ given there. Recall that $\mu_n = \tv{\alpha_n} + \beta_n$, and define analogously $\mu = \tv{\alpha} + \beta$, in particular we have
\begin{equation} \label{eq:conv-mu}
	\lim_{n \to +\infty} \mu_n(t) = \mu(t).
\end{equation}

Define also the measure $\widetilde \mu$ by
\[ \widetilde \mu(A) = \mu(A) + \int_{(0,\infty) \times A} \frac{x^2}{1+x^2} \nu(dx \, dy), \ A \in \Bcal. \]

We prove the four properties~\ref{property:triplet}--\ref{property:tightness} in Sections~\ref{sub:property-triplet} to~\ref{sub:property-tightness}.

\subsection{Proof of property~\ref{property:triplet}} \label{sub:property-triplet}

That $\Delta \alpha(t) \geq -1$ is a direct consequence of~\refAtwo\ since $\alpha_{i,n} \geq -1$ for every $i \geq 0$ and $n \geq 1$. Moreover, note that
\begin{equation} \label{eq:bound-proof-thm}
	\int_{(0,\infty) \times I_{s,t}} \frac{x^2}{2(1+x^2)} \nu(dx \, dy) \leq \beta(I_{s,t}), 0 \leq s \leq t,
\end{equation}
where $I_{s,t} = (s,t]$ or $I_{s,t} = (s,t)$. Indeed, for $I_{s,t} = (s,t]$
\begin{align*}
	\int_{(0,\infty) \times (s,t]} \frac{x^2}{2(1+x^2)} \nu(dx \, dy) & \mathop{=}^{\text{(i)}} \int_{(0,\infty)} \frac{x}{(1+x^2)^2} \nu([x, \infty) \times (s,t]) dx\\
	& \mathop{=}^{\text{(ii)}} \int_{(0,\infty)} \frac{x}{(1+x^2)^2} \liminf_{n \to +\infty} \nu_n([x, \infty) \times (s,t]) dx\\
	& \mathop{\leq}^{\text{(iii)}} \liminf_{n \to +\infty} \int_{(0,\infty)} \frac{x}{(1+x^2)^2} \nu_n([x, \infty) \times (s,t]) dx\\
	& \mathop{=}^{\text{(iv)}} \liminf_{n \to +\infty} \int_{(0,\infty) \times (s,t]} \frac{x^2}{2(1+x^2)} \nu_n(dx \, dy) \mathop{\leq}^{\text{(v)}} \liminf_{n \to +\infty} \left( \beta_n((s,t]) \right)
\end{align*}
using Fubini's theorem for~(i) and~(iv), the assumption~\refAone\ for~(ii) (using also that the set $\{ x: \nu(\{x\} \times (s,t]) > 0 \}$ has zero Lebesgue measure), Fatou's lemma for~(iii) and finally the definition of $\nu_n$ and $\beta_n$ for~(v). To get the result for $I_{s,t} = (s,t)$ apply~\eqref{eq:bound-proof-thm} to $(s,t']$ and let $t' \uparrow t$. The inequality~\eqref{eq:bound-proof-thm} has two direct consequences: $\int_{(0,\infty) \times (0,t]} (1 \wedge x^2) \nu(dx \, dy)$ is finite and the function $\widetilde \beta$ is increasing. Thus to conclude the proof of property~\ref{property:triplet}, it remains to show that $\widetilde \beta$ is continuous, i.e., $\Delta \widetilde \beta(t) = 0$.

First, by letting $s \uparrow t$ in~\eqref{eq:bound-proof-thm} with $I_{s,t} = (s,t]$ we see that $\Delta \widetilde \beta(t) = 0$ when $\Delta \beta(t) = 0$, so we only have to consider the case where $\Delta \beta(t) > 0$. In this case, the assumption~\refAtwo\ implies that $\beta_{\gamma_n(t), n} \to \Delta \beta(t)$ and that $\nu_{\gamma_n(t),n}([x,\infty)) \to \nu( [x, \infty) \times \{t\})$ for every $x>0$ such that $\nu(\{x \} \times \{t\})=0$. Then for any $d>0$ with $\nu(\{d\} \times \{t\}) = 0$, we get by weak convergence of probability measures (since all the measures restricted to $[d,\infty)$ have finite mass) and the dominated convergence theorem
\begin{equation} \label{eq:d}
	\lim_{n \to +\infty} \int_{(d, \infty)} \frac{x^2}{2(1+x^2)} \nu_{\gamma_n(t),n}(d x) = \int_{(d, \infty) \times \{t\}} \frac{x^2}{2(1+x^2)} \nu(dx \, dy).
\end{equation}

On the other hand, we have
\[ \int_{[-1/n, d]} \frac{x^2}{1+x^2} \nu_{\gamma_n(t),n}(d x) \leq \left( d + \frac{1}{n} \right) \int_{[-1/n, d]} \frac{\vert x\rvert}{1+x^2} \nu_{\gamma_n(t),n}(d x) \]
and from the definition of $\alpha_{\gamma_n(t), n}$ we see that
\begin{align*}
	\int_{[-1/n, d]} \frac{\lvert x\rvert}{1+x^2} \nu_{\gamma_n(t),n}(d x) & = \int_{[-1/n, d]} \frac{x}{1+x^2} \nu_{\gamma_n(t),n}(d x) + \frac{2/n}{1+(1/n)^2} \nu_{\gamma_n(t), n}\{-1/n\}\\
	& \leq \alpha_{\gamma_n(t),n} + \frac{2 n^2}{1+n^2}
\end{align*}
using that $\nu_{\gamma_n(t), n}\{-1/n\} \leq \nu_{\gamma_n(t), n}(\R) = n$ for the last inequality. Since $\lvert\alpha_{\gamma_n(t), n}\rvert \leq \tv{\alpha_n}(t)$ and $\tv{\alpha_n}(t) \to \tv{\alpha}(t)$, we obtain from the two last displays
\[
	\lim_{d \to 0} \limsup_{n \to +\infty} \int_{[-1/n, d]} \frac{x^2}{1+x^2} \nu_{\gamma_n(t),n}(d x) = 0.
\]

Combined with~\eqref{eq:d}, this gives
\begin{align*}
	\lim_{n \to +\infty} \int \frac{x^2}{2(1+x^2)} \nu_{\gamma_n(t),n}(d x) & = \lim_{d \to 0} \int_{(d, \infty) \times \{t\}} \frac{x^2}{2(1+x^2)} \nu(dx \, dy)\\
	& = \int_{(0, \infty) \times \{t\}} \frac{x^2}{2(1+x^2)} \nu(dx \, dy).
\end{align*}

Since $2 \beta_{\gamma_n(t), n} = \int \frac{x^2}{1+x^2} \nu_{\gamma_n(t),n}(d x)$ and $\beta_{\gamma_n(t), n} \to \Delta \beta(t)$ this concludes the proof of property~\ref{property:triplet}.

\subsection{Proof of property~\ref{property:u}} \label{sub:property-u}

The (classical) idea is to use a Lipschitz property satisfied by $\Psi$, combined with Gronwall's lemma. The Lipschitz property of $\Psi$ takes the following form (remember the constant $c_3$ defined in~\eqref{eq:c2+c3}): for any measurable and positive functions $f_1$ and $f_2$ and any $A \in \Bcal$, we have
\begin{equation} \label{eq:lipschitz}
	\left\lvert \Psi(f_1)(A) - \Psi(f_2)(A) \right\rvert \leq c_3\left( \inf_A f_1 \wedge \inf_A f_2 , \sup_A f_1 + \sup_A f_2 \right) \int_A \lvert f_1 - f_2\rvert d \widetilde \mu.
\end{equation}

Indeed, let $\eta = \inf_A f_1 \wedge \inf_A f_2$ and $T = \sup_A f_1 + \sup_A f_2$: then by definition of $\Psi$ we have
\begin{multline*}
	\left\lvert \Psi(f_1)(A) - \Psi(f_2)(A) \right\rvert \leq \int_A \lvert f_1-f_2\rvert d\tv{\alpha} + \int_A \left\lvert f_1^2-f_2^2\right\rvert d\beta\\
	+ \int_{(0,\infty) \times A} \left\lvert h(x, f_1(y)) - h(x, f_2(y)) \right\rvert \nu(dx\,dy).
\end{multline*}

Using $\lvert f_1^2 - f_2^2\rvert = \lvert f_1 - f_2\rvert (f_1 + f_2)$ and plugging in the constant $c_2$, we obtain
\begin{multline*}
	\left\lvert \Psi(f_1)(A) - \Psi(f_2)(A) \right\rvert \leq \int_A \lvert f_1-f_2\rvert d\tv{\alpha} + T \int_A \lvert f_1-f_2\rvert d\beta\\
	+ c_2(\eta, T) \int_{(0,\infty) \times A} \left\lvert f_1(y) - f_2(y) \right\rvert \frac{x^2}{1+x^2} \nu(dx\,dy) \leq c_3(\eta, T) \int_{A} \lvert f_1 - f_2\rvert d \widetilde \mu,
\end{multline*}
which establishes~\eqref{eq:lipschitz}. We will leverage this property using the following backwards version of Gronwall's lemma. The proof is standard and omitted.

\begin{Lem} \label{lemma:gronwall}
	Let $u$ and $R$ be non-negative, c\`adl\`ag functions and let $\pi$ be a locally finite and positive measure. If
	\[ u(s) \leq R(s) + \int_{(s,t]} u(x)\pi(dx) \]
	holds for all $0 \leq s \leq t$, then for all $0 \leq s \leq t$ we have
	\[ u(s) \leq R(s) + e^{\pi(s,t]} \int_{(s,t]} R(x)\pi(dx). \]
\end{Lem}

The property~\ref{property:u} of Theorem~\ref{thm:main} follows readily from Lemma~\ref{lemma:proof-thm-1} below. The proof of this lemma uses the following result, whose long proof is postponed to Appendix~\ref{appendix:proof-lemma}. Note that by considering a sequence $\ell_n \to \lambda$, Lemma~\ref{lemma:proof-thm-1} establishes more than what is necessary for property~\ref{property:u}. However, this generalization will be needed for the proof of property~\ref{property:fdd-convergence} in the next subsection.

\begin{Lem} \label{reste}
Fix $t, \lambda > 0$ and consider any sequence $(\ell_n)$ with $\ell_n \to \lambda$. For $n \geq 1$, let $R_n$ be the function
\[ R_n(s) = \left\vert \Psi_{n}(u_n(\, \cdot \,,t,\ell_n)) ((s,t]) - \Psi(u_n(\,\cdot\,,t,\ell_n))((s,t]) \right\vert, \ 0 \leq s \leq t. \]

Then $R_n(s) \to 0$ for any $\bottleneck(t) < s \leq t$ and $\sup \left\{ R_n(s): 0 \leq s \leq t, n \geq 1 \right\}$ is finite.
\end{Lem}

\begin{Lem} \label{lemma:proof-thm-1}
	Fix $t, \lambda > 0$ and a sequence $(\ell_n)$ with $\ell_n \to \lambda$. Then for any $s \in [\bottleneck(t), t]$, the sequence $(u_n(s,t,\ell_n), n \geq 1)$ converges and the function
	\[ u: s \in [\bottleneck(t), t] \mapsto \lim_{n \to +\infty} u_n(s,t,\ell_n) \]
	is the unique function satisfying the following properties:
	\begin{enumerate}
		\item $u(s) = \lambda + \Psi(u)((s,t])$ for all $\bottleneck(t) \leq s \leq t$;
		\item $u$ is c\`adl\`ag;
		\item $\inf_{[s, t]} u > 0$ for any $\bottleneck(t) < s \leq t$.
	\end{enumerate}
	
	In particular, $u(s,t,\lambda)$ does not depend on the sequence $\ell_n$.
\end{Lem}

\begin{proof}
	In the rest of the proof fix $t, \lambda > 0$ and $(\ell_n)$ a sequence converging to $\lambda$. Let $\ell = \inf_{n \geq 1} \ell_n$ and $L = \sup_{n \geq 1} \ell_n$ and assume without loss of generality, since $\ell_n \to \lambda > 0$, that $\ell > 0$. To ease the notation, we write in the rest of the proof $\bottleneck = \bottleneck(t)$ and $u_n(s) = u_n(s,t,\ell_n)$ for $0 \leq s \leq t$. We decompose the proof in four steps: first we prove that the sequence $(u_n(s), n \geq 1)$ is Cauchy for any $s \in (\bottleneck, t]$, then that it is Cauchy for $s = \bottleneck$, then that $u$ satisfies the claimed properties and finally that it is the only such function.
	
	Before beginning, note that everything is trivial if $\bottleneck = t$, because then $u_n(s) = \ell_n$ and $\Psi(u)((s,t]) = 0$ for any $s \in [\bottleneck, t]$. Hence in the sequel we assume that $\bottleneck < t$.
	\\
	
	\noindent \textit{First step: $(u_n(s))$ is Cauchy for $s \in (\bottleneck, t]$.} In the rest of this step fix $s \in (\bottleneck, t]$ and for $s \leq y \leq t$ we define as in Lemma~\ref{reste} $R_n(y) = \vert \Psi_n(u_n)((y,t]) - \Psi(u_n)((y,t]) \vert$. Then the second equality in~\eqref{eq:dynamics-u-n} gives for any $s \leq y \leq t$ and any $m,n \geq 1$
	\[ \left\lvert u_n(y) - u_m(y) \right\rvert \leq R_n(y) + R_m(y) + \left\lvert \Psi(u_n)((y,t]) - \Psi(u_m)((y,t]) \right\rvert. \]
	
	We get from~\eqref{eq:lipschitz}
	\[ \left \lvert \Psi(u_n)((y,t]) - \Psi(u_m)((y,t]) \right \rvert \leq c_3 \left( \inf_{(y,t]} u_n \wedge \inf_{(y,t]} u_m, \, \sup_{(y,t]} u_n + \sup_{(y,t]} u_m \right) \int_{(y,t]} \left \lvert u_n - u_m \right \rvert d \widetilde \mu. \]

	Since the function $u_n(s,t,\lambda)$ is increasing in $\lambda$, we have for any $y \in [s,t]$ and $n \geq N_{s,t,\ell}$ (recall that $N_{s,t,\ell}$ is defined in~\eqref{eq:def-underline-c^u} and is finite by Lemma~\ref{lemma:inf})
	\[ u_n(y) = u_n(y,t,\ell_n) \geq u_n(y,t,\ell) \geq \inf_{s \leq y' \leq t} u_n(y',t,\ell) \geq \underline c_{s,t,\ell}^u > 0. \]
	
	Similar monotonicity arguments lead to $u_n(y) \leq \overline c_{t, L}^u$ for any $y \leq t$ and $n \geq 1$, so that monotonicity properties of $c_3(\eta, T)$ in $\eta$ and $T$ give for $n,m \geq N_{s,t,\lambda}$
	\[ \left \lvert \Psi(u_n)((y,t]) - \Psi(u_m)((y,t]) \right \rvert \leq c_3 \left( \underline c_{s,t,\ell}^u, 2 \overline c_{t,L}^u \right) \int_{(y,t]} \left \lvert u_n - u_m \right \rvert d \widetilde \mu. \]
	
	We finally get the bound
	\[ \left \lvert u_n(y) - u_m(y) \right \rvert \leq R_n(y) + R_m(y) + C \int_{(y,t]} \left \lvert u_n - u_m \right \rvert d \widetilde \mu \]
	with $C = C_{s,t,\ell,L} = c_3(\underline c_{s,t,\ell}^u, 2\overline c_{t,L}^u)$, which holds for all $s \leq y \leq t$ and all $n, m \geq N_{s,t,\ell}$. Thus Lemma~\ref{lemma:gronwall} implies for those $n,m$
	\[ \left \lvert u_n(s) - u_m(s) \right \rvert \leq R_n(s) + R_m(s) + C e^{C \widetilde \mu(s,t]} \int_{(s,t]} \left( R_n + R_m \right) d\widetilde \mu \]
	so that for any $n_0 \geq N_{s,t,\ell}$,
	\[ \sup_{n,m \geq n_0} \left \lvert u_n(s) - u_m(s) \right \rvert \leq 2 \sup_{n \geq n_0} \left( R_n(s) \right) + 2 C e^{C \widetilde \mu(s,t]} \sup_{n \geq n_0} \left( \int_{(s,t]} R_n d\widetilde \mu \right). \]
	
	Lemma~\ref{reste} combined with the dominated convergence theorem shows that the right hand side of the above inequality goes to $0$ as $n_0 \to +\infty$ which proves that the sequence $(u_n(s), n \geq 1)$ is Cauchy and completes the proof of this first step.
	\\
	
	\noindent \textit{Second step: $(u_n(\bottleneck))$ is Cauchy.} For any $\bottleneck < s' \leq t$,~\eqref{eq:variation} entails
	\begin{align*}
		\left\lvert u_n(\bottleneck) - u_m(\bottleneck) \right \rvert & \leq \left\lvert u_n(\bottleneck) - u_n(s') \right \rvert + \left\lvert u_m(\bottleneck) - u_m(s') \right \rvert + \left\lvert u_n(s') - u_m(s') \right \rvert\\
		& \leq 2 \Delta_{t,\ell_n}^u \mu_n(\bottleneck, s'] + \left\lvert u_n(s') - u_m(s') \right \rvert\\
		& \leq 2 \Delta_{t,L}^u \mu_n(\bottleneck, s'] + \left\lvert u_n(s') - u_m(s') \right \rvert
	\end{align*}
	using for the last inequality that $\ell_n \leq L$ and that $\Delta_{t,y}^u$ is increasing in $y$, as can be seen directly from its definition~\eqref{eq:def-Delta}. Hence for any $n_0 \geq 1$,
	\[ \sup_{m,n \geq n_0} \left \lvert u_n(\bottleneck) - u_m(\bottleneck) \right \rvert \leq 2 \Delta_{t,L}^u \sup_{n \geq n_0} \mu_n(\bottleneck, s'] + \sup_{m,n \geq n_0} \left \lvert u_n(s') - u_m(s') \right \rvert. \]

	By~\eqref{eq:conv-mu} and the fact that $(u_n(s'))$ is Cauchy by the first step since $\bottleneck < s' \leq t$, the right hand side of the above inequality goes to $2 \Delta_{t,L}^u \mu(\bottleneck, s']$ as $n_0$ goes to infinity. Since $\mu(\bottleneck, s'] \to 0$ as $s' \downarrow \bottleneck$, letting then $s' \downarrow \bottleneck$ shows that $(u_n(\bottleneck))$ is Cauchy.
	\\
	
	\noindent \textit{Third step: properties of $u$.} Let from now on $u$ denote the function of the statement and consider $s \in [\bottleneck, t]$. First note that the second property follows from the first one, so we only have to prove the first and third ones. Assume first that $s > \bottleneck$. We have seen in the first step that for any $s \leq y \leq t$ and $n \geq N_{s,t,\ell}$
	\[ 0 < \underline c_{s,t,\ell}^u \leq u_n(y) \leq \overline c_{t,L}^u < +\infty. \]
	
	Since $u_n(y) \to u(y)$ for $s \leq y \leq t$ by definition of $u$, $u$ also satisfies $\underline c_{s,t,\lambda}^u \leq u(y) \leq \overline c_{t,L}^u$ for $s \leq y \leq t$. In particular the third property $\inf_{[s,t]} u > 0$ is satisfied. Let us now show the first property, still in the case $s > \bottleneck$. Plugging in~\eqref{eq:dynamics-u-n}, we get
	\begin{multline*}
		\left\lvert u(s) - \lambda - \Psi(u)((s,t]) \right \rvert \leq \lvert u(s) - u_n(s)\lvert + \left\lvert \Psi_n(u_n)((s,t]) - \Psi(u_n)((s,t]) \right \rvert\\
		+ \left\lvert \Psi(u_n)((s,t]) - \Psi(u)((s,t]) \right \rvert.
	\end{multline*}
	
	Since both $u_n$ and $u$ are bounded uniformly on $[s,t]$ by $\underline c_{s,t,\ell}^u$ and $\overline c_{t,L}^u$, we get with similar arguments as in the first step
	\[ \left \lvert \Psi(u_n)((s,t]) - \Psi(u)((s,t]) \right \rvert \leq c_3 \left( \underline c_{s,t,\ell}^u, 2 \overline c_{t,L}^u \right) \int_{(s,t]} \lvert u_n - u \lvert d \widetilde \mu \]
	and finally, we have for $n \geq N_{s,t,\ell}$
	\begin{multline*}
		\left\lvert u(s) - \lambda - \Psi(u)((s,t]) \right \rvert \leq \lvert u(s) - u_n(s) \lvert + \left\lvert \Psi_n(u_n)((s,t]) - \Psi(u_n)((s,t]) \right \rvert\\
		+ c_3 \left( \underline c_{s,t,\ell}^u, 2 \overline c_{t, L}^u \right) \int_{(s,t]} \lvert u_n - u \lvert d \widetilde \mu.
	\end{multline*}
	
	Let now $n$ go to infinity. The first term of the above upper bound goes to $0$ by definition of $u(s)$; the second term goes to $0$ by Lemma~\ref{reste}. Finally, the last term also goes to $0$ using the dominated convergence theorem. Thus $u$ satisfies the first property for $s > \bottleneck$.
	
	To extend this for $s = \bottleneck$, we proceed as in the second step and consider any $\bottleneck < s' \leq t$: then $\lvert u_n(\bottleneck) - u_n(s') \lvert \leq \Delta_{t,L}^u \mu_n(\bottleneck, s']$ and taking the limit $n \to +\infty$ gives $\lvert u(\bottleneck) - u(s') \lvert \leq \Delta_{t,L}^u \mu(\bottleneck, s']$.
	Letting $s' \downarrow \bottleneck$ shows that $u(s') \to u(\bottleneck)$. On the other hand, it is plain that $\lambda + \Psi(u)((s',t]) \to \lambda + \Psi(u)((\bottleneck, t])$ as $s' \downarrow \bottleneck$ and so $u$ satisfies the first property for all $s \in [\bottleneck, t]$. It remains to show uniqueness in order to complete the proof.
	\\
	
	\noindent \textit{Fourth step: uniqueness.} Let $\widetilde u$ be a function with the same properties than $u$. Then~\eqref{eq:lipschitz} gives
	\[ \lvert u(s) - \widetilde u(s) \lvert = \left\lvert \Psi(u)((s,t]) - \Psi(\widetilde u)((s,t]) \right \rvert \leq c_3 \left( \underline c_{s,t,\ell} \wedge \inf_{[s,t]} \widetilde u , \, \overline c_{t,L}^u + \sup_{[s,t]} \widetilde u \right) \int_{(s,t]} \lvert u - \widetilde u \lvert d \widetilde \mu \]
	and we conclude that $u = \widetilde u$ using Lemma~\ref{lemma:gronwall}.
\end{proof}

\subsection{Proof of property~\ref{property:fdd-convergence}} \label{sub:property-weak-convergence}

Fix in the rest of the proof $t \geq 0$, $s \in [\bottleneck(t), t]$, $x_n \to x \in [0,\infty)$, $I \geq 1$, $\lambda_1, \ldots, \lambda_I > 0$ and $s \leq t_1 < \cdots < t_I \leq t$. Consider first the case $I = 2$, so that we must show that
\begin{equation} \label{eq:conv-fdd}
	\lim_{n \to +\infty} \E \left( e^{-\lambda_1 X_n(t_1)-\lambda_2 X_n(t_2)} \mid X_n(s) = x_n \right) = \exp \left( - x u(s, t_1, \lambda_1 + u(t_1, t_2, \lambda_2)) \right).
\end{equation}

Using the Markov property of $X_n$ and the definition~\eqref{eq:def-u} of~$u_n$, we get
\begin{align*}
	\E \left( e^{-\lambda_1 X_n(t_1)-\lambda_2 X_n(t_2)} \mid X_n(s) = x_n \right) & = \E \left[ e^{-\lambda_1 X_n(t_1)} \E \left( e^{-\lambda_2 X_n(t_2)} \mid X_n(t_1) \right) \mid X_n(s) = x_n \right]\\
	 & = \E \left[ e^{-\lambda_1 X_n(t_1)} e^{-X_n(t_1) u_n(t_1, t_2, \lambda_2)} \mid X_n(s) = x_n \right]
\end{align*}
and so
\begin{equation} \label{eq:fd}
	\E \left( e^{-\lambda_1 X_n(t_1)-\lambda_2 X_n(t_2)} \mid X_n(s) = x_n \right) = \exp \left( - x_n u_n(s,t_1, \lambda_1 + u_n(t_1, t_2, \lambda_2)) \right).
\end{equation}

Since $\bottleneck$ is an increasing function by Lemma~\ref{lemma:inf} and $\bottleneck(t) \leq s \leq t_1 \leq t_2 \leq t$, it holds that $\bottleneck(t_2) \leq t_1 \leq t_2$ and so Lemma~\ref{lemma:proof-thm-1} implies that $u_n(t_1, t_2, \lambda_2) \to u(t_1, t_2, \lambda_2)$. Also, $\bottleneck(t_1) \leq s \leq t_1$ so Lemma~\ref{lemma:proof-thm-1} implies that
\[ \lim_{n \to +\infty} u_n(s, t_1, \lambda_1 + u_n(t_1, t_2, \lambda_2)) = u(s, t_1, \lambda_1 + u(t_1, t_2, \lambda_2)) \]
which proves~\eqref{eq:conv-fdd} using~\eqref{eq:fd}. The general case $I \geq 3$ follows in a similar way by induction.

\subsection{Proof of property~\ref{property:tightness}} \label{sub:property-tightness}

Fix $t \geq s \geq 0$, $x \geq 0$ and $x_n \to x$: the goal is to show that the sequence $(X_n(y), s \leq y \leq t)$ under $\P(\, \cdot \mid X_n(s) = x_n)$ is tight, and in order to do so we use Theorem~$1'$ in Bansaye and Simatos~\cite{Bansaye:1}. There are two assumptions, A1 and A2', to check.

Assumption A1 is a compact containment condition, and since $[0,\infty]$ endowed with the metric $d(x,y) = \lvert e^{-x} - e^{-y}\rvert$ is compact (in addition to being complete and separable), it is automatically checked. Thus we only have to check A2', i.e., we have to show that for each $n \geq 1$, each $s \leq y_0 \leq y \leq t$ with $\mu_n(y_0, y] \leq \Delta^u_{t,2} / 2$ and each $x_0 \in [0,\infty]$,
\begin{equation} \label{eq:goal-tightness}
	\E \left[ d(x_0, X_n(y))^2 \mid X_n(y_0) = x_0 \right] \leq 2 \Delta^u_{t,2} \mu_n(y_0, y].
\end{equation}

Indeed, in this case assumption~A2' is satisfied with $\eta_0 = (\Delta^u_{t,2})^2$, $F^n = 2 \Delta^u_{t,2} \mu_n$ and $F = 2 \Delta^u_{t,2} \mu$. So let us now prove~\eqref{eq:goal-tightness}. Since $\infty$ is, by convention, an absorbing state, we only need to prove this inequality for finite $x_0$. Starting from the left-hand side of~\eqref{eq:goal-tightness}, we get by expanding the square
\begin{align*}
	\E \left[ d(x_0, X_n(y))^2 \mid X_n(y_0) = x_0 \right] & = e^{-2x_0} + e^{-x_0 u_n(y_0, y, 2)} - 2 e^{- x_0 - x_0 u_n(y_0, y,1)}\\
	& = e^{-2x_0} \left[ 2 \left(1 - e^{x_0(1-u_n(y_0, y,1))} \right) - \left( 1 - e^{x_0(2-u_n(y_0, y,2))} \right) \right].
\end{align*}

Since $\lvert 1-e^z\rvert \leq e^{\lvert z \rvert}-1 \leq 2 (e^{\lvert z \rvert} - 1)$ for any $z \in \R$, we obtain further
\begin{align*}
	\E \left[ d(x_0, X_n(y))^2 \mid X_n(y_0) = x_0 \right] & \leq 2 e^{-2x_0} \left[ e^{x_0 \lvert 1 - u_n(y_0, y,1)\rvert} + e^{x_0 \lvert 2 - u_n(y_0, y,2)\rvert} - 2 \right].
\end{align*}

Writing $\lambda$ as $\lambda = u_n(y,y,\lambda)$,~\eqref{eq:variation} gives
\begin{align*}
	\E \left[ d(x_0, X_n(y))^2 \mid X_n(y_0) = x_0 \right] & \leq 2 e^{-2x_0} \left( e^{x_0 \Delta_{y,1}^u \mu_n(y_0, y]} + e^{x_0 \Delta_{y,2}^u \mu_n(y_0, y]} - 2 \right)\\
	& \leq 4 e^{-2x_0} \left( e^{x_0 \Delta_{t,2}^u \mu_n(y_0,y]} - 1 \right)
\end{align*}
using for the last inequality that $y \leq t$ and that both maps $z \mapsto \Delta^u_{z, \lambda}$ and $\lambda \mapsto \Delta^u_{z, \lambda}$ are increasing. Further, elementary analysis shows that for any $0 \leq y' \leq 1$
\[ \sup_{x' \geq 0} \left( e^{-x'} (e^{x' y'} - 1) \right) = y' \left( 1-y' \right)^{1/y'-1} \leq \frac{1}{e} y' e^{y'} \leq y'. \]

Combining the two last displays with $x' = 2 x_0$ and $y' = \Delta^u_{t,2} \mu_n(y_0,y]/2 \leq 1$, we finally get~\eqref{eq:goal-tightness} which achieves the proof of property~\ref{property:tightness}.

\section{Proof of Propositions~\ref{prop:no-bottleneck} and~\ref{prop:no-explosion}} \label{sec:proof-propositions}

\subsection{Proof of Proposition~\ref{prop:no-bottleneck}}

Fix some $t > 0$ and assume that $(\tv{\alpha_n}(t), n \geq 1)$ and $(\beta_n(t), n \geq 1)$ are bounded, and that~\eqref{eq:no-bottleneck} holds. In view of Lemma~\ref{lemma:inf}, in order to prove that $\bottleneck(t) = 0$ it is enough to prove that
\[ \liminf_{n \to +\infty} \inf_{0 \leq y \leq t} u_n(y, t, 1) > 0. \]

The goal is to apply the following lemma to derive a lower bound on $\inf_{0 \leq y \leq t} u_n(y, t, 1)$.

\begin{Lem} \label{lemmesuiterec}
	Let $\varepsilon > 0$, $M \geq 0$ and for each $i \geq 0$, $a_i, b_i \geq 0$ such that $a_i^2 - a_i b_i M \geq \varepsilon$. If $(w_i, 0 \leq i \leq I)$ satisfies $w_I > 0$, $0 \leq w_i \leq M$ for $0 \leq i \leq I$ and $w_i \geq w_{i+1} a_i - w_{i+1}^2 b_i$ for $0 \leq i \leq I-1$, then
	\[ w_i \geq \left(\frac{1}{w_I} \pi_{i, I-1} + \sum_{k=i}^{I-1} \pi_{i,k-1} b_k a_k^{-2} \right)^{-1}, \ 0 \leq i \leq I, \]
	where $\pi_{i,i-1} = 1$ and $\pi_{i,j} = \prod_{k=i}^j \big((1 + b_k^2 M^2\varepsilon^{-1}) a_k^{-1} \big)$ for $i \leq j$.
\end{Lem}

\begin{proof}
In the rest of the proof let $\rho_k = b_k^2 M^2 / \varepsilon$ and
	\[ r_i(x) = b_i^2 \frac{x^2}{a_i^2-a_i b_i x}, \ x \leq M. \]

	Note that $a_i^2 - a_i b_i M \geq \varepsilon$ by assumption, so $a_i>0$ and $r_i$ is well-defined and increasing. In particular, $r_i(x) \leq r_i(M) \leq b_i^2 M^2/\varepsilon = \rho_i$ so that writing
	\[ a_i w_{i+1} - b_i w_{i+1}^2 = a_i \left(\frac{1+r_i(w_{i+1})}{w_{i+1}} + \frac{b_i}{a_i} \right)^{-1} \]
	we obtain, using $w_i \geq a_i w_{i+1} - b_i w_{i+1}^2$,
	\[ w_i \geq a_i \left(\frac{1+\rho_i}{w_{i+1}} + \frac{b_i}{a_i} \right)^{-1}. \]
	
	This last inequality can be rewritten as $\overline w_i \leq (1+\rho_i)a_i^{-1} \overline w_{i+1} + b_i a_i^{-2}$ with $\overline w_i = 1/w_i$. It follows by induction that $\overline w_i \leq \overline w_I \pi_{i, I-1} + \sum_{k=i}^{I-1} \pi_{i,k-1} b_k / a_k^2$ which is exactly the desired result.
\end{proof}

Let $C = 1/(2 \overline c^u_{t,1})$ and
\[ \delta = \frac{1}{2} \min \left( \liminf_{n \to +\infty} \inf_{0 \leq i \leq I_n} \E\left( \xi_{i,n} ; \xi_{i,n} \leq nC + 1 \right), \ 1 \right), \]
so that $\delta > 0$ by~\eqref{eq:no-bottleneck}. We want to apply the previous lemma for $n$ large enough to $\varepsilon = \delta^2/16$, $M = \overline c^u_{t,1}$, $w_{i,n} = u_n(t^n_i, t, 1)$, $I_n = \gamma_n(t)$,
\[ a_{i,n} = 1 + (1+\epsilon_{i,n}(w_{i+1,n})) \int_{[-1/n, C]} x \nu_{i,n}(dx) \]
and
\[ b_{i,n} = (1+\epsilon_{i,n}(w_{i+1,n})) \int_{[-1/n, C]} x^2 \nu_{i,n}(dx). \]

By Lemma~\ref{bounded}, $M$ is finite and by definition, $w_{I_n, n} = 1 > 0$ and $0 \leq w_{i,n} \leq M$ for all $0 \leq i \leq I_n$. Since $1-e^{-x} \geq x-x^2$ for $x \geq -1$ and their exists a finite $n_0$ such that $1 + \epsilon_{i,n}(w_{i+1,n}) \geq 0$ for all $n \geq n_0$ and $i < I_n$ (by Lemmas~\ref{aspsin} and~\ref{bounded}), we obtain for all $n \geq \max(M, n_0)$ and $i < I_n$
\[ (1 + \epsilon_{i,n}(w_{i+1,n})) \int \left( 1 - e^{-x w_{i+1,n}} \right) \nu_{i,n}(dx) \geq w_{i+1,n} (a_{i,n}-1) - w_{i+1,n}^2 b_{i,n}. \]

Note that the left-hand side is by~\eqref{eq:def-epsilon} equal to $\psi_{i,n}(w_{i+1,n})$, and that $\psi_{i,n}(w_{i+1,n})$ is equal to $w_{i,n} - w_{i+1,n}$ according to the first equality in~\eqref{eq:def-psi-in} and the composition rule~\eqref{eq:composition-rule}. Thus we obtain
\begin{equation} \label{eq:w}
	w_{i,n} \geq w_{i+1,n} a_{i,n} - w_{i+1,n}^2 b_{i,n}, \ n \geq \max(M, n_0), i < I_n.
\end{equation}

In order to apply Lemma~\ref{lemmesuiterec} it remains to control the sequences $(a_{i,n})$ and $(b_{i,n})$. By definition of $\delta$, 
there exists a finite $n_1$ such that for all $n \geq n_1$ and $i \leq I_n$
\[ \int_{[-1/n, C]} x \nu_{i,n}(dx) = \E(\xi_{i,n} ; \xi_{i,n} \leq nC + 1) - \P(\xi_{i,n} \leq nC + 1) \geq \delta - 1. \]

Let $\eta > 0$ such that $1 + (1-\eta)(\delta-1) \geq \delta / 2$, and, according to Lemma~\ref{aspsin}, there exists $n_2$ such that $\lvert \epsilon_{i,n}(w_{i+1,n}) \rvert \leq \eta$ for all $n \geq n_2$ and $i < I_n$. Then from the definition of $a_{i,n}$ (and since $\delta < 1$) it follows that
\begin{equation} \label{eq:lower-bound-a}
	a_{i,n} \geq 1 + (1-\eta)(\delta-1) \geq \delta / 2, \ n \geq \max(n_1, n_2), i \leq I_n.
\end{equation}

We now proceed to controlling $a_{i,n}^2 - a_{i,n} b_{i,n} M$. First, we note that $\nu_{i,n}(\{-1/n\}) \leq n$ and
\[ \int_{[-1/n, C]} x^2 \nu_{i,n}(dx) = \frac{\nu_{i,n}(\{-1/n\})}{n^2} + \int_{[0, C]} x^2 \nu_{i,n}(dx) \leq \frac{1}{n} + C \int_{[0, C]} x \nu_{i,n}(dx) \]
and so we get $\int_{[-1/n, C]} x^2 \nu_{i,n}(dx) \leq 1/n + C + C \int_{[-1/n, C]} x \nu_{i,n}(dx)$. Then
\[ b_{i,n} \leq \frac{2}{n} + C (1+\epsilon_{i,n}(w_{i+1,n})) + C (a_{i,n}-1) \leq \kappa_n + C a_{i,n} \]
where $\kappa_n = 2/n + C \overline c^\epsilon_{n,t}(M)$. In particular, there exists by Lemma~\ref{aspsin} a finite $n_3$ such that $\kappa_n \leq \delta / (8M)$ for $n \geq n_3$, so that 
\[ a_{i,n}^2 - a_{i,n} b_{i,n} M \geq a_{i,n}^2 - a_{i,n} (\delta/(8M) + C a_{i,n}) M = \frac{1}{2} a_{i,n}^2 - \frac{\delta}{8} a_{i,n} \geq \frac{\delta^2}{16} \]
for $n \geq \max(n_1, n_2, n_3)$ and $i \leq I_n$ thanks to \eqref{eq:lower-bound-a}.
Thus for any $n \geq \max(M, n_0, n_1, n_2, n_3)$, the assumptions of Lemma~\ref{lemmesuiterec} are satisfied and we obtain
\[ w_{i,n} \geq \left(\pi_{i, I_n-1, n} + \sum_{k=i}^{I_n-1} \pi_{i,k-1, n} b_{k,n} a_{k,n}^{-2} \right)^{-1}, \ 0 \leq i \leq I_n, \]
where $\pi_{i,i-1,n} = 1$ and $\pi_{i,j,n} = \prod_{k=i}^j \big((1 + b_{k,n}^2 M^2\varepsilon^{-1}) a_{k,n}^{-1} \big)$ for $i \leq j$. To end the proof it remains to show that
\begin{equation} \label{eq:goal-limsup}
	\limsup_{n \geq 1} \sup_{0 \leq i \leq I_n} \left( \pi_{i, I_n-1, n} + \sum_{k=i}^{I_n-1} \pi_{i,k-1, n} b_{k,n} a_{k,n}^{-2} \right) < +\infty.
\end{equation}

Fix some $n \geq \max(M, n_0, n_1, n_2, n_3)$ and $0 \leq i \leq j < I_n$: we derive an upper bound on $\pi_{i,j,n}$, which we write as
\begin{equation} \label{eq:decomposition-pi}
	\pi_{i,j,n} = \exp \left( - \sum_{k=i}^j \log a_{k,n} \right) \times \prod_{k=i}^j \left(1+b_{k,n}^2 M^2 \varepsilon^{-1} \right).
\end{equation}

Let in the sequel $C' =  \sup \beta_n(t) + \sup \beta_n(t)^3$ and  the suprema are taken over $n \geq 1$ and note that
\[ \sum_{k=0}^{I_n-1} \beta_{k,n}^2 \leq \sum_{k=0}^{I_n-1} \beta_{k,n}\indicator{\beta_{k,n} \leq 1} +\beta_{n}(t)^2 \sum_{k=0}^{n-1} \indicator{\beta_{k,n} > 1} \leq C'.\]
Using convexity and $0 \leq b_{k,n} \leq 2 (1+C^2) \beta_{k,n}$, we get
\begin{equation} \label{eq:bound-product}
	 \prod_{k=i}^j \left(1+b_{k,n}^2 M^2 \varepsilon^{-1}\right) \leq \exp \left( \sum_{k=i}^j b_{k,n}^2 M^2 \varepsilon^{-1} \right) \leq \exp \left( 4 (1+C^2)^2 M^2 \varepsilon^{-1} C' \right).
\end{equation}

We now control the sum of the right-hand side of~\eqref{eq:decomposition-pi}. Since $n \geq \max(M, n_0)$ we have by~\eqref{eq:lower-bound-a} that $a_{k,n} \geq \delta / 2$ for $k < I_n$. In particular, if $\ell = \inf (\log x/ (x-1))$ where the infimum is taken over $x \geq \delta / 2$, $\ell \in (0,\infty)$ and we have $\log a_{k,n} \geq \ell (a_{k,n}-1)^-$, with $x^- = \min(x,0)$. Further,
\begin{align*}
	\int_{[-1/n, C]} x \nu_{k,n}(dx) & \geq -\frac{1}{n} \nu_{k,n}(\{-1/n\}) + \int_{[0, C]} \frac{x}{1+x^2} \nu_{k,n}(dx)\\
	& = -\frac{1}{n (n^2+1)} \nu_{k,n}(\{-1/n\}) + \alpha_{k,n}
 - \int_{(C,\infty)} \frac{x}{1+x^2} \nu_{k,n}(dx)\\
	& \geq - \lvert \alpha_{k,n} \rvert - (2 C^{-1}+1) \beta_{k,n}.
\end{align*}

Thus $\log a_{k,n} \geq - 2\ell\lvert \alpha_{k,n} \rvert - 2\ell (2 C^{-1}+1) \beta_{k,n}$ and summing over $k = i, \ldots, j$, we obtain 
that $\sum_{k=i}^j \log a_{k,n}$ is bounded. Recalling \eqref{eq:bound-product} and \eqref{eq:decomposition-pi}, we get that $\pi_{i,j,n}$ is bounded. Adding that
\[\sum_{k=0}^{I_n-1} \frac{b_{k,n}}{a_{k,n}^2} \leq \frac{8(1+C^2)}{\delta^2} \sum_{k=0}^{I_n-1} \beta_{k,n} \leq \frac{8(1+C^2)}{\delta^2} C', \]
the proof of~\eqref{eq:goal-limsup} and thus of Proposition~\ref{prop:no-bottleneck} is finally complete.

\subsection{Proof of Proposition~\ref{prop:no-explosion}} Let in the rest of the proof $\widetilde \nu_n$ be the following increasing, c\`adl\`ag function
\[ \widetilde \nu_n(t) = \int_{\R \times (0,t]} \lvert x \rvert \nu_n(dx\,dy) = n \sum_{i=0}^{\gamma_n(t)-1} \E (\lvert \overline \xi_{i,n}\rvert). \]

Assume that~\eqref{eq:condition-first-moment} holds, i.e., $\sup_{n \geq 1} \widetilde \nu_n(t) < +\infty$: we first show that it implies the two other assumptions. That the sequences $(\tv{\alpha_n}(t), n \geq 1)$ and $(\beta_n(t), n \geq 1)$ are bounded comes from~\eqref{eq:condition-first-moment} by summing from $i = 0$ to $\gamma_n(t)-1$ the two following inequalities:
\[ \vert \alpha_{i,n} \vert \leq n \E \left( \frac{\vert \overline
\xi_{i,n} \vert }{1+\overline \xi_{i,n}^2} \right) \leq  n \E \left(
\vert \overline \xi_{i,n} \vert  \right) \ \text{ and } \ \beta_{i,n} = \frac{1}{2} n \E \left( \frac{\overline
\xi_{i,n}^2}{1+\overline \xi_{i,n}^2} \right) \leq \frac{n}{2} \E \left( \vert
\overline \xi_{i,n} \vert  \right). \]

We now show that~\eqref{eq:tightness} also holds. By Lemma~\ref{bounded} there exists a finite constant $C_t$ such that $u_n(s,y,\lambda)\leq C_t$ for all $y \in [s,t]$ and $\lambda \leq 1$. Further, by Lemma~\ref{aspsin} there exists $n_t$ such that $\vert \epsilon_{i,n} (v)\vert \leq 1$ for any $n \geq n_t$, $v \leq C_t$ and $i\leq \gamma_n(t)$. Finally, invoking Lemma~\ref{relun} and using $1-\exp(-\lambda x)\leq \lambda \vert x \vert$ for $x\in\R$ and $\lambda \geq 0$, we get 
\[ u_n(s,t,\lambda) \leq \lambda + 2 \sum_{i=\gamma_n(s)+1}^ {\gamma_n(t)} u_n(t_{i}^n,t,\lambda) \int \vert x \vert\nu_{i-1,n}(dx) = \lambda + \int_{(s,t]} u_n(y, t, \lambda) \widetilde \nu_n(dy). \]

Thus Lemma~\ref{lemma:gronwall} implies that $u_n(s,t,\lambda) \leq \lambda + \lambda \widetilde \nu_n(s,t] e^{\widetilde \nu_n(s,t]}$, and consequently
\begin{equation} \label{eq:ttmp}
	\sup_{n \geq 1, s \leq y \leq t} u_n(s,y,\lambda) \leq \lambda \left[ 1 + \sup_{n \geq 1} \widetilde \nu_n(t) \exp \left( \sup_{n \geq 1} \widetilde \nu_n(t) \right) \right].
\end{equation}

Since $\sup_{n \geq 1} \widetilde \nu_n(t)$ is finite, letting $\lambda \to 0$ in~\eqref{eq:ttmp} we see that $\sup u_n(s,y,\lambda) \to 0$ as $\lambda \to 0$, where the supremum is taken over $n \geq 1$ and $s \leq y \leq t$. To see that this implies~\eqref{eq:tightness}, we only have to write for any $A \geq 1$
\[ \P\left(X_n(y) \geq A \ \vert \ X_n(s)=1\right) = \P\left(1-e^{-X_n(y)/A} \geq 1-1/e \ \vert \ X_n(s)=1 \right)\leq \frac{1-e^{-u_n(s,y,1/A)}}{1-1/e}. \]

\medskip

We now assume that $(\tv{\alpha_n}(t), n \geq 1)$ and $(\beta_n(t), n \geq 1)$ are bounded and that~\eqref{eq:tightness} holds, and we show that there exists
$n(k)\rightarrow \infty$ such that $\lim_{n \to \infty} u_{n(k)}(s,t,\lambda) = 0$ for every $s < \bottleneck(t)$. First, note that
\begin{equation} \label{limitecomp}
	\lim_{\lambda \to 0} \left( \sup_{n \geq 1, \, 0\leq s \leq y \leq t}u_n(s,y,\lambda) \right) = 0.
\end{equation}

Indeed, for any $0\leq s \leq y \leq t$ and $A > 0$, we can write
\[ 1 - e^{-u_n(s,y,\lambda)} = \E\left[1-\exp(-\lambda X_n(y)) \ \vert \ X_n(s)=1 \right] \leq 1-e^{-\lambda A} + \P(X_n(y)\geq A \ \vert \ X_n(s)=1) \]
which gives
\[ \sup_{n \geq 1, 0\leq s \leq y \leq t} u_n(s, y, \lambda) \leq - \log \left( e^{-\lambda A} - \sup_{n \geq 1, 0 \leq s \leq y \leq t} \P(X_n(y) \geq A \ \vert \ X_n(s)=1) \right). \]

Letting first $\lambda \to 0$ and then $A \to +\infty$ and using~\eqref{eq:tightness} gives~\eqref{limitecomp}. Further, Lemma~\ref{lemma:inf} guarantees the existence of sequences $(n(k))$ and $(y_k)$ such that $y_k \rightarrow \bottleneck (t)$, $n(k) \to \infty$ and $u_{n(k)}(y_k, t, \lambda) \to 0$ as $k \to +\infty$. Then, the composition rule~\eqref{eq:composition-rule} shows that for every $k \geq 1$ and $s\leq y_k$,
\[ u_{n(k)} \left( s, t, \lambda \right) = u_{n(k)} \left( s, y_k, u_{n(k)} \left( y_k, t, \lambda \right) \right) \leq \sup_{n \geq 1, s \leq y \leq t} u_n \left( s, y, u_{n(k)} \left( y_k, t, \lambda \right) \right). \]

Since $u_{n(k)}(y_k, t, \lambda) \to 0$ as $k \to +\infty$, \eqref{limitecomp} implies that $\sup\{u_{n(k)} \left( s, t, \lambda \right) : s\leq y_k\}\to 0$ which achieves to prove that $ u_{n(k)}(s,t,\lambda) \rightarrow 0$ for every $s<\bottleneck(t)$.

\appendix

\section{Proof of Lemma~\ref{lemma:equivalence-GW}} \label{appendix:proof-GW}

Assume that ${\offdistr}_{i,n} = {\offdistr}_{0,n}$ and that $\gamma_n(t) = \lfloor \Gamma_n t \rfloor$ for some sequence $\Gamma_n \to +\infty$. Define ${\offdistr}_n = {\offdistr}_{0,n}$ and $\xi_n = \xi_{0,n}$, and for each $n \geq 1$ let $(\overline \xi_n(k), k, \geq 1)$ be i.i.d.\ random variables distributed as $\overline \xi_n$. Then according to Theorem~$3.1$ in Grimvall~\cite{Grimvall74:0}, $(X_n)$ converges in the sense of finite-dimensional distribution if and only if the sequence of random variables $(\sum_{i=1}^{n \Gamma_n} \overline \xi_n(k), n \geq 1)$ converges weakly.

Moreover, since $q_{i,n} = q_n$ and $\gamma_n(t) = \lfloor \Gamma_n t \rfloor$, we have
\[ \alpha_n(t) = n \lfloor \Gamma_n t \rfloor \E \left( \frac{\overline \xi_n}{1 + \overline \xi_n^2} \right), \ \beta_n(t) = \frac{1}{2} n \lfloor \Gamma_n t \rfloor \E \left( \frac{\overline \xi_n^2}{1 + \overline \xi_n^2} \right) \]
and $\nu_n([x,\infty) \times (0,t]) = n \lfloor \Gamma_n t \rfloor \P(\overline \xi_n \geq x)$. In this context Assumption~\ref{assumptions} is therefore equivalent to the following assumption.

\begin{assumption}
	There exist $a \in \R$, $b \geq 0$ and a positive, $\sigma$-finite measure $F$ with support in $(0,\infty)$ such that
	\begin{multline} \tag{B1} \label{eq:assumption-GW}
		n \Gamma_n \E \left( \frac{\overline \xi_n}{1 + \overline \xi_n^2} \right) \mathop{\longrightarrow}_{n \to +\infty} a, \ n \Gamma_n \E \left( \frac{\overline \xi_n^2}{1 + \overline \xi_n^2} \right) \mathop{\longrightarrow}_{n \to +\infty} b,\\
		\text{and } \ n \Gamma_n \P(\overline \xi_n \geq x) \mathop{\longrightarrow}_{n \to +\infty} F([x,\infty))
	\end{multline}
	where the last convergence holds for every $x > 0$ such that $F(\{x\}) = 0$.
\end{assumption}

Under this assumption we then have $\alpha(t) = at$, $\beta(t) = bt$ and $\nu(dx \, dt) = dt F(dx)$. Thus to prove Lemma~\ref{lemma:equivalence-GW} we have to prove that~\eqref{eq:assumption-GW} is equivalent to the weak convergence of the sum $\sum_{k=1}^{n \Gamma_n} \overline \xi_n(k)$.
\\

Assume that~\eqref{eq:assumption-GW} holds. Then by~\ref{property:fdd-convergence} of Theorem~\ref{thm:main}, the sequence $(X_n, n \geq 1)$ converges in the sense of finite-dimensional distributions. By Grimvall~\cite[Theorem~$3.1$]{Grimvall74:0}, this implies the weak convergence of $\sum_{k=1}^{n \Gamma_n} \overline \xi_n(k)$.
\\

Assume now that $\sum_{k=1}^{n \Gamma_n} \overline \xi_n(k)$ converges weakly. Then Theorem~$1$ of \S~$25$ in Gnedenko and Kolmogorov~\cite{Gnedenko68:0} immediately gives the existence of $F$ with $\int (1\wedge x^2) F(dx) < +\infty$ such that the last convergence in~\eqref{eq:assumption-GW} holds. Let us prove the two first convergences of \eqref{eq:assumption-GW}. In the rest of the proof let $G \subset (0,\infty)$ denote the set of continuity points of $F$, for $\kappa = 1$ or $2$ let $m_\kappa(x) = \lvert x \lvert ^\kappa/(1+x^2)$ and fix some $\kappa \in\{1,2\}$.

Since $n \Gamma_n \P(\overline \xi_n > x) \to F((x,\infty))$ for $x \in G$, it follows that
\[ n \Gamma_n \E\left(m_\kappa(\overline \xi_n); \lvert \overline \xi_n\lvert > \varepsilon \right) \mathop{\longrightarrow}_{n \to +\infty} \int_{x > \varepsilon} m_\kappa(x) F(dx) \]
for any $\varepsilon \in G$. This can for instance be seen by considering the weak convergence of the random variables $\overline \xi_n$ conditioned on $\lvert \overline \xi_n \lvert > \varepsilon$. Moreover, according to Corollary~$15.16$ in Kallenberg~\cite{Kallenberg02:0}, there exists a finite number $d_\kappa$ such that $n \Gamma_n \E(\xi_n^\kappa; \lvert \overline \xi_n \lvert \leq \varepsilon) \to d_\kappa + L_\kappa(\varepsilon)$ for any $\varepsilon \leq 1$ in $G$, and where $L_1(\varepsilon) = - \int_{\varepsilon < x \leq 1} x F(dx)$ and $L_2(\varepsilon) = \int_{x \leq \varepsilon} x^2 F(dx)$. Note in particular that
\[ \sup_{n \geq 1} n \Gamma_n \E\left(m_2(\overline \xi_n)\right) < +\infty, \]
since
\begin{align*}
	n \Gamma_n \E\left(m_2(\overline \xi_n) \right) & = n \Gamma_n \E\left(m_2(\overline \xi_n) ; \overline \xi_n \leq \varepsilon \right) + n \Gamma_n \E\left(m_2(\overline \xi_n) ; \overline \xi_n > \varepsilon \right)\\
	& \leq n \Gamma_n \E\left(\overline \xi_n^2 ; \overline \xi_n \leq \varepsilon \right) + n \Gamma_n \E\left(m_2(\overline \xi_n) ; \lvert \overline \xi_n \rvert > \varepsilon \right)
\end{align*}
(note that $\overline \xi_n > \varepsilon \Leftrightarrow \lvert \overline \xi_n \rvert > \varepsilon$ for $n$ large enough, since $\overline \xi_n \geq -1/n$). We now write
\begin{align*}
	\E \left( \frac{\overline \xi_n^\kappa}{1+\overline \xi_n^2} \right) & = \E \left( \frac{\overline \xi_n^\kappa}{1+\overline \xi_n^2} - \overline \xi_n^\kappa ; \lvert \overline \xi_n \rvert \leq \varepsilon \right) + \E \left( \frac{\overline \xi_n^\kappa}{1+\overline \xi_n^2} ; \lvert \overline \xi_n \rvert > \varepsilon \right) + \E \left( \overline \xi_n^\kappa ; \lvert \overline \xi_n \rvert \leq \varepsilon \right)\\
	& = -\E \left( \frac{\overline \xi_n^{\kappa+2}}{1+\overline \xi_n^2} ; \lvert \overline \xi_n \rvert \leq \varepsilon \right) + \E \left( m_\kappa(\overline \xi_n) ; \lvert \overline \xi_n \rvert > \varepsilon \right) + \E \left( \overline \xi_n^\kappa ; \lvert \overline \xi_n \rvert \leq \varepsilon \right)
\end{align*}
to obtain
\begin{multline*}
	\left\lvert n \Gamma_n \E \left( \frac{\overline \xi_n^\kappa}{1+\overline \xi_n^2} \right) - n \Gamma_n \E \left( m_\kappa(\overline \xi_n) ; \lvert\overline \xi_n\rvert > \varepsilon \right) - n \Gamma_n \E \left( \overline \xi_n^\kappa ; \lvert \overline \xi_n \rvert \leq \varepsilon \right) \right\rvert\\
	\leq \varepsilon^\kappa n \Gamma_n \E\left(m_2(\overline \xi_n)\right) \leq \varepsilon^\kappa \sup_{n \geq 1} n \Gamma_n \E\left(m_2(\overline \xi_n)\right).
\end{multline*}

Letting first $n \to +\infty$ with $\varepsilon \in G$ and then $\varepsilon \to 0$, we thus obtain
\[ \lim_{\varepsilon \to 0} \limsup_{n \to +\infty} \left\lvert n \Gamma_n \E \left( \frac{\overline \xi_n^\kappa}{1+\overline \xi_n^2} \right) - \int_{x > \varepsilon} m_\kappa(x) F(dx) - d_\kappa - L_\kappa(\varepsilon) \right\rvert = 0. \]

For $\kappa = 2$ we have
\[ \int_{x > \varepsilon} m_2(x) F(dx) - d_2 - L_2(\varepsilon) \mathop{\longrightarrow}_{\varepsilon \to 0} \int m_2(x) F(dx) \]
which is finite, while for $\kappa = 1$ we have by definition of $L_1$
\[ \int_{x > \varepsilon} m_1(x) F(dx) - d_1 - L_1(\varepsilon) = \int_{\varepsilon < x \leq 1}(m_1(x) - x) F(dx) - d_1 + \int_{x > 1} m_1(x) F(dx). \]

Since $m_1(x) - x \sim -x^2$ as $x \to 0$, $\int_{\varepsilon < x \leq 1}(m_1(x) - x) F(dx)$ converges by the dominated convergence theorem to $\int_{x \leq 1}(m_1(x) - x) F(dx)$ as $\varepsilon$, which is finite since $\int (1 \wedge x^2) F(dx) < +\infty$. This completes the proof.

\section{Proof that the key constants are finite} \label{appendix:constants}

\begin{Lem}[Control of $c_1$, $c_2$ and $c_3$] \label{lemma:c_1+c_2+c_2}
	For any $C \geq 0$ and $0 < \eta < T$, the constants $c_1(C)$, $c_2(\eta, T)$ and $c_3(\eta, T)$ are finite.
\end{Lem}

\begin{proof}
	Since $\lim_{x\to+\infty} \Phi_1(x) = 0$, $\Phi_1$ is bounded on $[-C,+\infty)$ for any $C > 0$ and so the constant $c'_1(C)$ of~\eqref{eq:c_1} is finite in view of~\eqref{eq:g-h-Phi}. In particular $c_1(C)$ is finite for every $C \geq 0$. Let $0 < \eta < T$ and $\eta \leq y, y' \leq T$, and fix $x \geq 0$: $c_2(\eta, T)$ is finite because
	\[ \left\lvert \frac{h(x, y) - h(x, y')}{(y - y') x^2 / (1+x^2)} \right\rvert \leq \sup_{\eta \leq v \leq T} \left\lvert H_x'(v) \right\rvert, \]
	with $H_x(y) = h(x, y) (1+x^2) / x^2$. One can compute $H_x'(y) = x e^{-y x} + y \Phi_1(y x)$ and so
	\[ \sup_{\eta \leq v \leq T} \left\lvert H_x'(v) \right\rvert \leq \frac{1}{\eta} \sup_{v \geq 0} (v e^{-v}) + T \sup_{v \geq 0} \lvert\Phi_1(v)\rvert. \]
	
	This upper bound being independent of $x$, we get the finiteness of $c_2(\eta, T)$, and hence of $c_3(\eta, T)$.
\end{proof}

\begin{Lem}[Control of $\overline c_{n,t}^\epsilon(C)$] \label{aspsin}
	Fix $t \geq 0$ and assume that the sequences $(\tv{\alpha_n}(t), n \geq 1)$ and $(\beta_n(t), n \geq 1)$ are bounded. Then for any $C \geq 0$, we have $\overline c_{n,t}^\epsilon(C) \to 0$ as $n$ goes to infinity.
\end{Lem}

\begin{proof}
	Fix $t$ and $C \geq 0$ and define
	\[ I_{t, C} = \sup \left\{ \left\lvert\int (1 - e^{-\lambda x}) \nu_{i,n}(d x) \right\rvert : 1 \leq n, 0 \leq i < \gamma_n(t), 0 \leq \lambda \leq C \right\}. \]
	Then~\eqref{eq:ineq-g-2} entails
	\begin{multline*}
		I_{t,C} \leq c_1(C) \sup \left\{ \mu_n(t_{i}^n, t_{i+1}^n]: n \geq 1, 0 \leq i < \gamma_n(t) \right\}\\
		\leq c_1(C) \sup_{n \geq 1} \left( \sum_{i=0}^{\gamma_n(t)-1} \mu_n(t_{i}^n, t_{i+1}^n] \right) = c_1(C) \sup_{n \geq 1} \mu_n(t).
	\end{multline*}

	Since $\mu_n(t) = \tv{\alpha_n}(t) + \beta_n(t)$ the sequence $(\mu_n(t))$ is bounded by assumption, showing that $I_{t,C}$ is finite. It follows from the definitions~\eqref{eq:def-psi-in} and~\eqref{eq:def-epsilon} of $\psi_{i,n}$ and $\epsilon_{i,n}$ that for any $i \geq 0$
\[ \epsilon_{i,n}(\lambda) = \frac{-\log\left(1-\frac{1}{n} \int (1 - e^{-\lambda x}) \nu_{i,n}(d x) \right) - \frac{1}{n} \int (1 - e^{-\lambda x}) \nu_{i,n}(d x)}{\frac{1}{n}\int (1 - e^{-\lambda x}) \nu_{i,n}(d x)} \]
and so
\[ \overline c_{n, t}^\epsilon(C) \leq \sup_{\lvert x\rvert \leq I_{t,C} / n} \left\lvert \frac{-\log(1-x) - x}{x} \right\rvert. \]

Letting $n \to +\infty$ achieves the proof of the result.
\end{proof}

\begin{Lem}[Control of $\overline c_{t, \lambda}^u$] \label{bounded}
	Fix $t \geq 0$ and assume that the two sequences $(\tv{\alpha_n}(t), n \geq 1)$ and $(\beta_n(t), n \geq 1)$ are bounded. Then for any $\lambda \geq 0$ the constant $\overline c_{t, \lambda}^u$ is finite and moreover
\[ \sup \left\{ \overline c_{s, \lambda}^u: 0 \leq s \leq t, \lambda \leq 1 \right\} < + \infty. \]
\end{Lem}

\begin{proof}
	In the rest of the proof fix $t$ and $\lambda \geq 0$, define $B_t = 2\sup_{n \geq 1} \mu_n(t)$, which is finite by assumption, and $C_{t,\lambda} = (\lambda + 2)(1 + B_t) e^{B_t}$. Following Lemma~\ref{aspsin} choose $n_{t,\lambda} \geq 1$ such that $\overline c_{n,t}^\epsilon(C_{t,\lambda}) \leq 1$ for all $n \geq n_{t, \lambda}$. Since $Z_{i,n}$ is finite for each $i \geq 0$ and $n \geq 1$, it follows that $\sup_{0 \leq s \leq t} u_n(s,t,\lambda)$ is finite for each $n \geq 1$. Thus to prove the result, it is enough to prove that
	\[ \sup \left \{ u_n(s,t,\lambda): n \geq n_{t,\lambda}, 0 \leq s \leq t \right \} = \sup \left \{ u_n(t_i^n,t_{\gamma_n(t)}^n,\lambda): n \geq n_{t,\lambda}, 0 \leq i \leq \gamma_n(t) \right \} \]
	is finite. In the rest of the proof fix $n \geq n_{t,\lambda}$ and define $a_i = u_n(t_i^n,t_{\gamma_n(t)}^n,\lambda)$. We prove by backwards induction on $i$ that $a_i \leq C_{t,\lambda}$ for all $0 \leq i \leq \gamma_n(t)$, and since the bound does not depend on $n$ or $i$ this will show the result. We have $a_{\gamma_n(t)} = \lambda \leq C_{t,\lambda}$ so the initialization is satisfied. Now consider some $1 \leq i < \gamma_n(t)$ and assume that $a_k \leq C_{t,\lambda}$ for all $i \leq k \leq \gamma_n(t)$: we prove that $a_{i-1} \leq C_{t, \lambda}$.
	
	Fix some $i < k \leq \gamma_n(t)$. By definition, we have
	\[ \psi_{k-1,n}(a_k) = (1+\epsilon_{k-1,n}(a_k)) \left( a_k \alpha_{k-1,n} + \int g(x, a_k) \nu_{k-1,n} (dx) \right). \]
	
	By induction hypothesis, it holds that $a_k \leq C_{t,\lambda}$. Combined with $\overline c_{n,t}^\epsilon(C_{t,\lambda}) \leq 1$, this gives $0 \leq 1 + \epsilon_{k-1,n}(a_k) \leq 2$. Together with the inequality $g(x,y) \leq x^2 / (1+x^2)$ (note that $\Phi_1 \geq 0$), we finally get
	\[
		\psi_{k-1,n}(a_k) \leq (1+\epsilon_{k-1,n}(a_k)) \left( a_k \lvert\alpha_{k-1,n}\rvert + 2 \beta_{k-1,n} \right) \leq 2 (a_k+2) \mu_n(t_{k-1}^n, t_{k}^n].
	\]
	
	Hence for any $i-1 \leq j \leq \gamma_n(t)$, this gives together with Lemma~\ref{relun} for the first equality
	\[
		a_{j} = \lambda + \sum_{k=j+1}^{\gamma_n(t)} \psi_{k-1,n}(a_k) \leq \lambda + \sum_{k=j+1}^{\gamma_n(t)} 2 (a_k+2) \mu_n(t_{k-1}^n, t_{k}^n].
	\]
	
	This can be rewritten $a'_j \leq A + S_{j+1}$ if $a_k' = a_k + 2$, $A = \lambda + 2$, $d_k = 2 \mu_n(t_{k-1}^n, t_{k}^n]$ and $S_k = d_k a_k' + \cdots + d_{\gamma_n(t)} a_{\gamma_n(t)}'$. This gives for $j = i-1$
	\[ a'_{i-1} \leq A + S_i = A + d_{i} a'_{i} + S_{i+1} \leq A + d_{i} (A + S_{i+1}) + S_{i+1} = (1+d_{i})(A + S_{i+1}). \]
	
	Then by induction one gets
	\[ a'_{i-1} \leq (1+d_{i}) \cdots (1+d_{\gamma_n(t)-1}) (A + S_{\gamma_n(t)}) \leq \exp \left(d_{1} + \cdots + d_{\gamma_n(t)} \right) (A + d_{\gamma_n(t)} a_{\gamma_n(t)}'). \]
	
	Since $a'_{\gamma_n(t)} = A = \lambda + 2$ and $d_{\gamma_n(t)} \leq d_1 + \cdots + d_{\gamma_n(t)} = 2 \mu_n(t) \leq B_t$, this shows that $a_{i-1} \leq C_{t, \lambda}$ which achieves the proof of the induction and shows that $\overline c^u_{t,\lambda} \leq C_{t, \lambda}$. This gives the finiteness of $\overline c^u_{t,\lambda}$, and since $C_{t, \lambda}$ is increasing in both $t$ and $\lambda$, for any $s \leq t$ and $\lambda \leq 1$ we obtain $\overline c^u_{s, \lambda} \leq C_{t, 1}$ which gives the second part of the lemma.
\end{proof}

\section{Proof of Lemma~\ref{reste}} \label{appendix:proof-lemma}

This appendix is devoted to the proof of Lemma~\ref{reste}. Recall the function $\mu = \tv{\alpha} + \beta$ defined at the beginning of Section~\ref{proofThm1}. We will use the following simple result.

\begin{Lem} \label{decompmesure}
	For any $\varepsilon > 0$ and $0 \leq s < t$, there exists a partition of the interval $(s,t]$ as
	\[ (s,t] = \left( \bigcup_{j = 1}^J (a_j,b_j] \right) \cup \left( \bigcup_{k = 1}^K (a'_k,b'_{k}] \right) \]
	such that $\{ b'_k, 1 \leq k \leq K \} = (s,t] \cap \{ v \geq 0: \Delta \mu(v) \geq \varepsilon \}$, $\mu(a_j, b_j] \leq \varepsilon$ for each $1 \leq j \leq J$ and $\mu(a'_k, b_k') \leq \varepsilon / K$ for each $1 \leq k \leq K$.
\end{Lem}


In the rest of the proof fix $t, \lambda > 0$, $\bottleneck(t) < s \leq t$, $(\ell_n)$ a sequence converging to $\lambda$ and let $u_n(y) = u_n(y,t,\ell_n)$. With this notation, we have
\[ R_n(y) = \left\vert \Psi_{n}(u_n) ((y,t]) - \Psi(u_n)((y,t]) \right\vert, \ 0 \leq y \leq t. \]

Let $\ell = \inf_{n \geq 1} \ell_n$ and $L = \sup_{n \geq 1} \ell_n$ and assume without loss of generality, since $\ell_n \to \lambda > 0$, that $\ell > 0$. We first show that $R_n(s) \to 0$, the fact that $\sup \{ R_n(y): s \leq y \leq t, n \geq 1 \}$ is finite is proved in Section~\ref{subsub:last}. From the definitions of $\Psi$ and $\Psi_n$ one can write
\[ \left\vert \Psi_{n}(u_n) ((s,t]) - \Psi(u_n)((s,t]) \right\vert \leq B_n^{\alpha} + B_n^{\beta} + B_n^{\nu} + B_n^{\epsilon} \]
with
\[ B_n^{\alpha} = \left\vert \int_{(s,t]} u_n d \alpha_n - \int_{(s,t]} u_n d \alpha \right\vert, \quad B_n^{\beta} = \left\vert \int_{(s,t]} u_n^2 d \beta_n - \int_{(s,t]} u_n^2 d \beta \right\vert, \]
\[ B_n^{\nu} = \left\vert \int_{[-1/n,\infty) \times (s,t]} h(x,u_n(y)) \nu_n(dx \, dy) - \int_{(0,\infty) \times (s,t]} h(x,u_n(y)) \nu(dx \, dy)\right\vert \]
and
\[ B_n^{\epsilon} = \sum_{i = \gamma_n(s)+1}^{\gamma_n(t)} \left \lvert \epsilon_{i-1,n}(u_n(t_i^n)) \right \rvert \left \lvert \int \left( 1-e^{-x u_n(t_i^n)} \right) \nu_{i-1,n}(dx) \right \rvert. \]

We will show that each sequence $(B_n^\alpha)$, $(B_n^\beta)$, $(B_n^\nu)$ and $(B_n^\epsilon)$ goes to $0$ as $n$ goes to infinity. By~\eqref{eq:ineq-g-2} and by definition of the constants $\overline c_{n,t}^\epsilon$, $\overline c_{t,L}^u$ and $c_1$, one can derive similarly as in the proof of~\eqref{eq:variation}
\[ B_n^\epsilon \leq \overline c_{n,t}^\epsilon ( \overline c_{t, \ell_n}^u ) c_1(\overline c_{t,\ell_n}^u) \mu_n(t) \leq \overline c_{n,t}^\epsilon ( \overline c_{t, L}^u ) c_1(\overline c_{t,L}^u) \mu_n(t) \]
where the last inequality follows from the fact that $\ell_n \leq L$ and that the functions $\overline c_{n,t}^\epsilon(C)$ and $\overline c_{t,y}^u$ are increasing in $C$ and $y$, respectively. From now on, we will use such monotonicity properties without further comment. This last upper bound is seen to go $0$, invoking~\eqref{eq:conv-mu} and Lemmas~\ref{aspsin} and~\ref{bounded}. Thus the sequence $(B_n^\epsilon)$ goes to $0$ and we have to control the three other sequences $(B_n^\alpha)$, $(B_n^\beta)$ and $(B_n^\nu)$. We control the two first sequences in Section~\ref{subsub:alpha+beta} and the last one in Section~\ref{subsub:nu}

\subsection{Control of the sequences $(B_n^\alpha)$ and $(B_n^\beta)$} \label{subsub:alpha+beta}

We treat in detail the convergence of $(B_n^\alpha)$ to $0$. For $(B_n^\beta)$, one essentially needs to replace $\alpha$ by $\beta$ and $u_n$ by $u_n^2$, we mention along the way what modifications need to be done.

Fix $\varepsilon > 0$ and consider the partition $((a_j, b_j], 1 \leq j \leq J)$ and $((a_k', b_k'], 1 \leq k \leq K)$ of $(s,t]$ provided by Lemma~\ref{decompmesure}. Note that the partition depends on $s,t$ and $\varepsilon$ but not on $n$. We can write $B_n^{\alpha} \leq \sum_{j=1}^{J} B_{n,j}^{\alpha,1} + \sum_{k=1}^{K} ( B_{n,k}^{\alpha,2} + B_{n,k}^{\alpha,3})$ with
\[
	B_{n,j}^{\alpha,1} = \left\vert \int_{(a_{j},b_{j}]} u_n d \alpha_n - \int_{(a_{j},b_{j}]} u_n d \alpha \right\vert, \quad B_{n,k}^{\alpha,2} = \int_{(a'_{k},b'_{k})} u_n d \tv{\alpha_n} + \int_{(a'_{k},b'_{k})} u_n d \tv{\alpha}
\]
and $B_{n,k}^{\alpha,3} = u_n(b'_{k}) \big\lvert \alpha_{\gamma_n(b'_{k}),n} - \Delta \alpha(b'_k)\big \rvert $. For $B_{n,j}^{\alpha,1}$ we have
\begin{multline*}
	B_{n,j}^{\alpha,1} \leq \int_{(a_{j},b_{j}]} \left \lvert u_n(y) - u_n(b_{j}) \right \rvert \tv{\alpha_n}(dy) + \int_{(a_{j},b_{j}]} \left \lvert u_n(y) - u_n(b_{j}) \right \rvert \tv{\alpha}(dy)\\
	+ u_n(b_{j}) \left \lvert \alpha_n(a_{j}, b_{j}] - \alpha(a_{j}, b_{j}] \right \rvert.
\end{multline*}

By~\eqref{eq:variation}, $\left \lvert u_n(y) - u_n(b_{j}) \right \rvert \leq \Delta_{t, L}^u \mu_n(a_j,b_j]$ for all $y \in (a_j, b_j]$ and so, using also $u_n(b_{j}) \leq \overline c^u_{t, L}$, we get
\[
	B_{n,j}^{\alpha,1} \leq \Delta_{t, L}^u \mu_n(a_j,b_j] \big( \tv{\alpha_n} (a_j,b_j] + \tv{\alpha} (a_j,b_j] \big) + \overline c^u_{t, L} \left \lvert \alpha_n(a_{j}, b_{j}] - \alpha(a_{j}, b_{j}] \right \rvert.
\]

For $B_n^{\beta,1}$ one needs to use
\[ \left \lvert u_n(y)^2 - u_n(b_{j})^2 \right \rvert = \left \lvert u_n(y) - u_n(b_{j}) \right \rvert \left( u_n(y) + u_n(b_{j}) \right) \leq 2 \overline c_{t,L}^u \Delta_{t, L}^u \mu_n(a_j,b_j], \]
which leads to a similar upper bound. Since the partition does not depend on $n$, we have $\alpha_n(a_j, b_j]\to \alpha(a_j, b_j]$ and $\mu_n(a_j,b_j] \to \mu(a_j,b_j]$ by~\refAone, so that summing over $j = 1, \ldots, J$, letting $n$ go to infinity and using $\tv{\alpha}(A) \leq \mu(A)$ gives
\begin{equation} \label{eq:bound-1-1}
	\limsup_{n \to +\infty} \sum_{j=1}^J B_{n,j}^{\alpha,1} \leq 2 \Delta_{t, L}^u \sum_{j=1}^J \big( \mu(a_{j}, b_{j}] \big)^2 \leq 2 \varepsilon \Delta_{t, L}^u \mu(s,t],
\end{equation}
using also $\mu(a_j, b_j] \leq \varepsilon$, which holds by choice of the partition, to derive the second inequality. To upper bound $B_{n,k}^{\alpha,2}$ we write $B_{n,k}^{\alpha,2} \leq \overline c_{t, L}^u \left(\tv{\alpha_n}(a_k',b_k') + \tv{\alpha}(a_k',b_k')\right)$ which leads, using $\mu(a'_k, b'_k) \leq \varepsilon / K$, to
\begin{equation} \label{eq:bound-1-2}
	\limsup_{n \to +\infty} \sum_{k=1}^K B_{n,k}^{\alpha,2} \leq 2 \overline c_{t, L}^u \sum_{k=1}^K \mu(a'_k, b'_k) \leq 2 \varepsilon \overline c_{t, L}^u.
\end{equation}

For $B_{n,k}^{\beta,2}$ one can use $B_{n,k}^{\beta,2} \leq (\overline c_{t, L}^u)^2 \left(\beta_n(a_k',b_k') + \beta(a_k',b_k')\right)$ to obtain a similar upper bound. Finally, for $B_{n,k}^{\alpha,3}$ one has $B_{n,k}^{\alpha,3} \leq \overline c_{t, L}^u \vert \alpha_{\gamma_n(b'_k),n} - \Delta \alpha(b'_k) \vert$ which goes to $0$ by~\refAtwo. One can similarly write $B_{n,k}^{\beta,3} \leq (\overline c_{t, L}^u)^2 \vert \beta_{\gamma_n(b'_k),n} - \Delta \beta(b'_k) \vert$ for $B_{n,k}^{\beta, 3}$. Since $K$ does not depend on $n$ this gives $\sum_{k=1}^K B_{n,k}^{\alpha,3} \to 0$ and so~\eqref{eq:bound-1-1} and~\eqref{eq:bound-1-2} give
\[ \limsup_{n \to +\infty} B_n^{\alpha} \leq 2 \varepsilon \left( \Delta_{t, L}^u \mu(s,t] + \overline c_{t, L}^u \right). \]

Since $\varepsilon$ was arbitrary, letting $\varepsilon \to 0$ gives the result.

\subsection{Control of the sequence $(B_n^\nu)$} \label{subsub:nu}

For $T \geq 0$ we define the constant
\begin{equation}
\label{nvlconst}
c_4(T) = \sup \left\{ \left \lvert \frac{h(x,y)}{x^3 / (1+x^2)} \right \rvert : x \geq -1, 0 \leq y \leq T \right\} 
\end{equation}
which, starting from~\eqref{eq:g-h-Phi}, can be seen to be finite. For $d > 0$ we write
\begin{equation} \label{eq:bound-3}
	B_n^{\nu} \leq \widetilde B_n^{\nu} + \hat{B}_n^{\nu} + \check B_n^{\nu}
\end{equation}
with
\[ \widetilde B_n^{\nu} = \left\vert \int_{[d,\infty) \times (s,t]} h(x,u_n(y)) \nu_n (dx \, dy) - \int_{[d,\infty) \times (s,t]} h(x,u_n(y)) \nu (dx \, dy)\right\vert. \]
$$ \hat B_n^{\nu}= \int_{[-1/n,d) \times (s,t]} \lvert h(x,u_n(y)) \rvert \nu_n (dx \, dy), \quad \check B_n^{\nu}= \int_{(0,d) \times (s,t]} \lvert h(x,u_n(y)) \rvert \nu (dx \, dy). $$	

Note that $\widetilde B_n^{\nu}$ depends on $d$ but, similarly as $t$ or $\lambda$, we do not reflect this in the notation because $d$ will be fixed once and for all shortly. Bounding the two last terms thanks to~\eqref{nvlconst}, we have
\[ B_n^{\nu} \leq \widetilde B_n^{\nu} + c_4(\overline c_{t,L}^u) \left( \int_{(0,d) \times (0,t]} \frac{x^3}{1+x^2} \nu (dx \, dy) + \int_{[-1/n, d) \times (0,t]} \frac{\lvert x \rvert ^3}{1+x^2} \nu_{n}(dx \, dy)\right). \]

Since $\int_{(0,\infty) \times (0,t]} (1 \wedge x^2) \nu(dx \, dy)$ is finite, we have $\int_{(0,d) \times (0,t]} \frac{x^3}{1+x^2} \nu (dx \, dy) \to 0$ as $d \to 0$. Moreover, proceeding similarly as for the proof of~\eqref{eq:d}, we can show that
\[ \lim_{d \to 0} \limsup_{ n \to +\infty} \int_{[-1/n, d) \times (0,t]} \frac{\lvert x \rvert ^3}{1+x^2} \nu_n(dx \, dy) = 0. \]

Thus letting first $n \to +\infty$ and then $d \to 0$, we obtain
\[ \limsup_{n \to +\infty} B_n^{\nu} \leq \lim_{d \to 0} \limsup_{n \to +\infty} \widetilde B_n^{\nu}. \]

Hence to prove $B_n^{\nu} \to 0$ we only have to show that $\widetilde B_n^{\nu} \to 0$ for every $d > 0$. So in the rest of this step we fix an arbitrary $d > 0$ and show that $\widetilde B_n^\nu \to 0$. Fix $\varepsilon > 0$ and consider the partition $((a_j, b_j], 1 \leq j \leq J)$ and $((a_k', b_k'], 1 \leq k \leq K)$ of $(s,t]$ given by Lemma~\ref{decompmesure}, which does not depend on $n$. Then we can write $\widetilde B_n^{\nu} \leq \sum_{j=1}^J \widetilde B_{n,j}^{\nu,1} + \sum_{k=1}^K ( \widetilde B_{n,k}^{\nu,2} + \widetilde B_{n,k}^{\nu,3})$ with
\begin{align*}
	\widetilde B_{n,j}^{\nu,1} & = \left\vert \int_{[d,\infty) \times (a_j, b_j]} h(x,u_n(y)) \nu_{n}(dx \, dy) - \int_{[d,\infty) \times (a_j, b_j]} h(x,u_n(y)) \nu (dx \, dy) \right\vert, \\
\widetilde B_{n,k}^{\nu,2} & = \int_{[d,\infty) \times (a_k, b_k')} \lvert h(x,u_n(y)) \rvert \nu_{n}(dx \, dy) + \int_{[d,\infty) \times (a_k, b_k')} \lvert h(x,u_n(y)) \rvert \nu (dx \, dy)
\end{align*}
and
\[
	\widetilde B_{n,k}^{\nu,3} = \left \lvert \int_{[d,\infty)} h(x,u_n(b'_k)) \nu_{\gamma_n(b'_k),n}(dx) - \int_{[d,\infty) \times \{b'_k\}} h(x,u_n(b'_k)) \nu (dx \, dy) \right \rvert.
\]

Further we write $\widetilde B_{n,j}^{\nu,1} \leq \widetilde B_{n,j}^{\nu,4} + \widetilde B_{n,j}^{\nu,5}$ with
\begin{multline*}
	\widetilde B_{n,j}^{\nu,4} = \int_{[d,\infty) \times (a_j, b_j]} \left \lvert h(x,u_n(y)) - h(x,u_n(b_j)) \right \rvert \nu_{n}(dx \, dy)\\
	+ \int_{[d,\infty) \times (a_j, b_j]} \left \lvert h(x,u_n(y)) - h(x,u_n(b_j)) \right \rvert \nu (dx \, dy)
\end{multline*}
and
\[ \widetilde B_{n,j}^{\nu,5} = \left\vert \int_{[d,\infty) \times (a_j, b_j]} h(x,u_n(b_j)) \nu_{n}(dx \, dy) - \int_{[d,\infty) \times (a_j, b_j]} h(x,u_n(b_j)) \nu (dx \, dy) \right\vert. \]
We derive, in order, upper bounds on $\widetilde B_{n,k}^{\nu,2}$, $\widetilde B_{n,j}^{\nu,4}$, $\widetilde B_{n,k}^{\nu,3}$ and finally on $\widetilde B_{n,j}^{\nu,5}$.
\\

To control $\widetilde B_{n,k}^{\nu,2}$ we introduce the constant
\[
	c_5(T) = \sup \left\{ \frac{\lvert h(x, y) \rvert }{x^2 / (1+x^2)}: 0 \leq y \leq T, x \geq 0 \right\}
\]
which can be seen to be finite, starting from instance from~\eqref{eq:g-h-Phi}. Thus
\begin{align*}
	\widetilde B_{n,k}^{\nu,2} & \leq c_5(\overline c_{t,L}^u) \left( \int_{[d,\infty) \times (a'_k, b'_k)} \frac{x^2}{1+x^2} \nu_{n}(dx \, dy) + \int_{[d,\infty) \times (a_k, b_k')} \frac{x^2}{1+x^2} \nu (dx \, dy) \right)\\
	& \leq 2 c_5(\overline c_{t,L}^u) \left( \beta_{n}(a'_k, b'_k) + \beta(a'_k, b'_k) \right)
\end{align*}
using~\eqref{eq:bound-proof-thm} for the last inequality. Using $\beta_n(a'_k, b'_k) \to \beta(a'_k, b'_k) \leq \mu(a'_k, b'_k)\leq \varepsilon / K$ (the convergence $\beta_n(a'_k, b'_k) \to \beta(a'_k, b'_k)$ comes from Assumptions~\refAone\ and~\refAtwo\ by writing $\beta_n(a'_k, b'_k) = \beta_n(a'_k, b'_k] - \Delta \beta_n(b'_k)$ and observing that $b'_k$ is by construction an atom of $\mu$), this leads to
\begin{equation} \label{eq:bound-nu-2}
	\limsup_{n \to +\infty} \sum_{k=1}^K \widetilde B_{n,k}^{\nu,2} \leq 4 \varepsilon c_5(\overline c_{t,L}^u).
\end{equation}

To derive an upper bound on $\widetilde B_{n,j}^{\nu,4}$, we use the constant $c_2(\eta, T)$ defined in~\eqref{eq:c2+c3}. Since $0 < \underline c_{s, t, \ell}^u \leq u_n(y) \leq \overline c_{t, L}^u$ for $n \geq N_{s,t,\ell}$ and $a_j < y \leq b_j$, we have for such $n$
\begin{multline*}
	\int_{[d,\infty) \times (a_j, b_j]} \left \lvert h(x,u_n(y)) - h(x,u_n(b_j)) \right \rvert \nu_{n}(dx \, dy)\\
	\leq c_2 \left(\underline c_{s, t, \ell}^u, \overline c_{t, L}^u \right) \int_{[d,\infty) \times (a_j, b_j]} \left \lvert u_n(y) - u_n(b_j) \right \rvert \frac{x^2}{1+x^2} \nu_{n}(dx \, dy).
\end{multline*}

Since $\lvert u_n(y) - u_n(b_j) \rvert \leq \Delta_{t,L}^u \, \mu_n(a_j, b_j]$ for $a_j < y \leq b_j$ by~\eqref{eq:variation}, we obtain
\begin{multline*}
	\int_{[d,\infty) \times (a_j, b_j]} \left \lvert h(x,u_n(y)) - h(x,u_n(b_j)) \right \rvert \nu_{n}(dx \, dy)\\
	\leq c_2 \left(\underline c_{s, t, \ell}^u, \overline c_{t, L}^u \right) \Delta_{t,L}^u \, \mu_n(a_j, b_j] \int_{[d,\infty) \times (a_j, b_j]} \frac{x^2}{1+x^2} \nu_{n}(dx \, dy).
\end{multline*}

Since
\begin{equation} \label{eq:aa}
	\int_{[d,\infty) \times (a_j, b_j]} \frac{x^2}{1+x^2} \nu_{n}(dx \, dy) \leq 2 \beta_n(a_j, b_j] \leq 2 \mu_n(a_j, b_j],
\end{equation}
we finally get
\[
	\int_{[d,\infty) \times (a_j, b_j]} \left \lvert h(x,u_n(y)) - h(x,u_n(b_j)) \right \rvert \nu_{n}(dx \, dy) \leq C_{s,t,\ell,L} (\mu_n(a_j, b_j])^2
\]
with $C_{s,t,\ell,L} = 2 c_2(\underline c_{s, t, \ell}^u, \overline c_{t, L}^u) \Delta_{t,L}^u$. The exact same reasoning with $\nu$ instead of $\nu_n$, using the inequality~\eqref{eq:bound-proof-thm} instead of~\eqref{eq:aa}, leads to
\[ \widetilde B_{n,j}^{\nu,4} \leq C_{s,t,\ell,L} \left[ (\mu_n(a_j,b_j])^2 + (\mu(a_j,b_j])^2 \right]. \]

Hence~\eqref{eq:conv-mu} gives
\begin{equation} \label{eq:bound-nu-1}
	\limsup_{n \to +\infty} \sum_{j=1}^J \widetilde B_{n,j}^{\nu,4} \leq 2 C_{s,t,\ell,L} \sum_{j=1}^J (\mu(a_j,b_j])^2 \leq 2 \varepsilon C_{s,t,\ell,L} \mu(s,t]
\end{equation}
using $\mu(a_j,b_j] \leq \varepsilon$ to get the second inequality.
\\

The arguments to control $\widetilde B_{n,k}^{\nu,3}$ and $\widetilde B_{n,j}^{\nu,5}$ are very similar: we treat the case $\widetilde B_{n,j}^{\nu,5}$ in detail and mention necessary changes needed for $\widetilde B_{n,k}^{\nu,3}$. We need the constant $c_6$
\begin{equation} \label{eq:c5}
	c_6(T) = \sup_{\substack{0 \leq y \leq T \\ 0 \leq x, x'}} \left \lvert \frac{h(x,y) - h(x',y)}{x-x'} \right \rvert
\end{equation}
which is finite because
\[ \frac{\partial h}{\partial x}(x,y) = y e^{-xy} + y \frac{x^2 + xy - 1}{(1+x^2)^2} \]
and so for $x, x' \geq 0$ and $0 \leq y \leq T$,
\[ \left \lvert \frac{\partial h}{\partial x}(x,y) \right \rvert \leq T + T \sup_{v \geq 0} \left( \frac{v^2 + Tv + 1}{(1+v^2)^2} \right). \]

Let $\pi_{n,j}$ be the signed measure defined for $A \in \Bcal$ by
\[ \pi_{n,j}(A) = \nu_{n}(A \times (a_j,b_j]) - \nu (A \times (a_j,b_j]). \]

For $\widetilde B_{n,k}^{\nu,3}$ one needs to consider the measure $\pi_{n,k}$ defined similarly but with $A \times \{b'_k\}$ instead of $A \times (a_j, b_j]$. With this notation we have
\[
	\widetilde B_{n,j}^{\nu,5} \leq \sup_{0 \leq y \leq \overline c_{t, L}^u} \left\vert \int_{[d,\infty)} h(x,y) \pi_{n,j}(dx) \right\vert.
\]

Fix $Y, \eta > 0$ and consider a subdivision $d = \tau_1 < \cdots < \tau_N < \tau_{N+1} = \infty$ with the following three properties: $(1)$ $\tau_{\ell+1} - \tau_\ell \leq \eta$ for all $1 \leq \ell < N$; $(2)$ $\tau_N = Y$; and $(3)$ $\nu(\{\tau_\ell\} \times (a_j, b_j]) = 0$ for all $1 \leq \ell \leq N$. For $\widetilde B_{n,k}^{\nu,3}$ the third condition should be $\nu(\{\tau_\ell\} \times \{b'_\ell\}) = 0$ for all $1 \leq \ell \leq N$. Then for any $y \geq 0$,
\begin{multline*}
	\left\vert \int_{[d,\infty)} h(x,y) \pi_{n,j}(dx) \right\vert \leq \sum_{\ell=1}^{N-1} \int_{[\tau_\ell,\tau_{\ell+1})} \lvert h(x,y) - h(\tau_\ell,y) \rvert \tv{\pi_{n,j}}(dx)\\
	+ \int_{[Y,\infty)} \lvert h(x,y) - h(Y,y) \rvert \tv{\pi_{n,j}} (dx) + \sum_{\ell=1}^N \lvert h(\tau_\ell,y) \rvert \left\vert \pi_{n,j}([\tau_\ell, \tau_{\ell+1})) \right\vert.
\end{multline*}

By choice of the partition $(\tau_\ell)$ and by definition~\eqref{eq:c5} of $c_6$, we have for any $y \leq \overline c^u_{t, L}$
\begin{align*}
	\sum_{\ell=1}^{N-1} \int_{[\tau_\ell, \tau_{\ell+1})} \lvert h(x,y) - h(\tau_\ell,y) \rvert \tv{\pi_{n,j}}(dx) & \leq c_6(\overline c^u_{t, L}) \sum_{\ell=1}^{N-1} \int_{[\tau_\ell, \tau_{\ell+1})} \lvert x - \tau_\ell \rvert \tv{\pi_{n,j}}(dx)\\
	& \leq \eta c_6(\overline c_{t, L}^u) \tv{\pi_{n,j}}([d,\infty)).
\end{align*}

Thus introducing the constant
\[ \overline c_{t,L}^h = \sup \left\{ \lvert h(x,y) \rvert : x \geq 0, 0 \leq y \leq \overline c_{t,L}^u \right\} \]
which in view of~\eqref{eq:g-h-Phi} can be seen to be finite, one gets for any $y \leq \overline c^u_{t, L}$,
\begin{multline*}
	\left\vert \int_{[d,\infty)} h(x,y) \pi_{n,j}(dx) \right\vert \leq \eta c_6(\overline c_{t, L}^u) \tv{\pi_{n,j}}([d,\infty)) + 2 \overline c_{t,L}^h \tv{\pi_{n,j}}([Y,\infty))\\
	+ \overline c_{t,L}^h \sum_{\ell=1}^N\left\vert \pi_{n,j}([\tau_\ell, \tau_{\ell+1})) \right\vert.
\end{multline*}

Since no $(\tau_\ell)$ is an atom of the measure $\int_{\, \cdot \, \times (a_j, b_j]} \nu(dx \, dy)$, it follows from~\refAone\ that $\pi_{n,j}([\tau_\ell, \tau_{\ell+1})) \to 0$ as $n$ goes to infinity for each $\ell$. Moreover, one has
\[ \lvert \pi_{n,j}(A) \rvert \leq \nu_{n}(A \times (a_j,b_j]) + \nu (A \times (a_j,b_j]) \]
and finally, for any $\eta > 0$ we have, using also the fact that $\limsup_{n\rightarrow\infty}\tv{\pi_{n,j}}([c,\infty)) \leq 2 \nu([c,\infty) \times (a_j, b_j])$ for any $c \geq 0$,
\begin{multline*}
	\limsup_{n \to +\infty} \sup_{0 \leq y \leq \overline c_{t, L}^u} \left\vert \int_{[d,\infty)} h(x,y) \pi_{n,j}(dx) \right\vert \leq 2 \eta c_6 (\overline c_{t,L}^u) \nu ([d,\infty) \times (a_j,b_j])\\
	+ 4 \overline c_{t,L}^h \nu([Y, \infty) \times (a_j, b_j]).
\end{multline*}

Thus letting $\eta \to 0$ and $Y \to +\infty$ finally shows that $\widetilde B_{n,j}^{\nu,5} \to 0$ for each $1 \leq j \leq J$ and also $\widetilde B_{n,k}^{\nu,3} \to 0$ for each $1 \leq k \leq K$. Hence combining~\eqref{eq:bound-nu-2} and~\eqref{eq:bound-nu-1} finally gives
\[ \limsup_{n \to +\infty} \widetilde B_n^\nu \leq \varepsilon \left[ c_5(\overline c_{t,L}^u) + 2 C_{s,t,\lambda} \mu(s,t] \right] \]
and since $\varepsilon$ is arbitrary, letting $\varepsilon \to 0$ achieves to prove that $R_n(s) \to 0$.

\subsection{Boundedness of $(R_n(y))$} \label{subsub:last}

We now complete the proof of the lemma by showing that $\sup \{ R_n(y): 0 \leq y \leq t, n \geq 1\}$ is finite. We have $R_n(y) \leq \left \lvert \Psi_n(u_n)((y,t]) \right \rvert + \left \lvert \Psi(u_n)((y,t]) \right \rvert$, so that it is enough to prove that
\begin{equation} \label{eq:sup-Psi-n}
	\sup \left\{ \left \lvert \Psi_n(u_n)((y,t]) \right \rvert: 0 \leq y \leq t, n \geq 1 \right\} < +\infty
\end{equation}
and similarly with $\Psi$ instead of $\Psi_n$. Using~\eqref{eq:dynamics-u-n} for the first equality and~\eqref{eq:variation} for the second inequality, we get for any $0 \leq y \leq t$
\[ \left \lvert \Psi_n(u_n)((y,t]) \right \rvert = \left \lvert u_n(y) - u_n(t) \right \rvert \leq \Delta_{t,L}^u \, \mu_n (y,t] \leq \Delta_{t,L}^u \, \sup_{n \geq 1} \mu_n (t) \]
so that~\eqref{eq:sup-Psi-n} holds. On the other hand, starting from the definition of $\Psi$ we get
\begin{align*}
	\left \lvert \Psi(u_n)((y,t]) \right \rvert & \leq \int_{(s,t]} \lvert u_n \rvert d\tv{\alpha} + \int_{(s,t]} u_n^2 d\beta + \int_{(0,\infty) \times (s,t]} \lvert h(x, u_n(y)) \rvert \nu(dx \, dy)\\
	& \leq \overline c_{t,L}^u \tv{\alpha} (t) + (\overline c_{t,L}^u)^2 \beta (t) + c_5(\overline c_{t, L}^u) \int_{(0,\infty) \times (0,t]} \frac{x^2}{1+x^2} \nu(dx \, dy)
\end{align*}
which ends the proof of the lemma, since this upper bound is finite (invoking~\eqref{eq:bound-proof-thm} for the finiteness of the integral term).

\bibliographystyle{plain}

\end{document}